\documentclass[review]{elsarticle}

\usepackage{lineno,hyperref}
\modulolinenumbers[5]
\journal{Journal of Computational Physics}

\usepackage{lipsum}
\usepackage{amsmath,amsfonts,amsthm,mathrsfs,amssymb}
\usepackage{graphicx}
\usepackage{epstopdf}
\usepackage{color}
\usepackage{algorithmic}
\newcommand{\R}{{\mathbb R}}

\newcommand{\C}{{\mathbb C}}

\newcommand{\bb}{\boldsymbol}

\newcommand{\be}{\begin{eqnarray}}
\newcommand{\ben}{\begin{eqnarray*}}
\newcommand{\en}{\end{eqnarray}}
\newcommand{\enn}{\end{eqnarray*}}

\newcommand{\pa}{\partial}

\newcommand{\ov}{\overline}

\newcommand{\G}{\Gamma}

\newcommand{\Om}{\Omega}

\newtheorem{theorem}{Theorem}[section]
\newtheorem{lemma}[theorem]{Lemma}

\definecolor{rot}{rgb}{0.000,0.000,0.000}

\begin{document}

\begin{frontmatter}

\title{An accurate boundary element method for the exterior elastic scattering problem in two dimensions}

\author[zju]{Gang Bao}
\ead{baog@zju.edu.cn}
\author[uestc,cq]{Liwei Xu}
\ead{xul@uestc.edu.cn}
\author[zju]{Tao Yin\corref{cor1}}
\ead{taoyin@zju.edu.cn}

\cortext[cor1]{Corresponding author.}
\address[zju]{School of Mathematical Sciences, Zhejiang University, Hangzhou 310027, China}
\address[uestc]{\textcolor{rot}{School of Mathematical Sciences, University of Electronic Science and Technology of China, Chengdu, Sichuan 611731, China}}
\address[cq]{\textcolor{rot}{Institute of Computing and Data Sciences, Chongqing University, Chongqing 401331, China}}

\begin{abstract}
This paper is concerned with a Galerkin boundary element method solving the two dimensional exterior elastic wave scattering problem. The original problem is first reduced to the so-called Burton-Miller (\cite{BM71}) boundary integral formulation, and essential mathematical features of its variational form are discussed. In numerical implementations, a newly-derived and analytically  accurate regularization formula (\cite{YHX}) is employed for the numerical evaluation  of hyper-singular boundary integral operator. A new computational approach is employed based on the series expansions of Hankel functions for the computation of weakly-singular boundary integral operators during the reduction of corresponding Galerkin equations into \textcolor{rot}{a} discrete linear \textcolor{rot}{system}. The effectiveness of  proposed numerical methods is demonstrated using several numerical examples.
\end{abstract}


\end{frontmatter}

\section{Introduction}
\label{sec:1}
The elastic wave scattering problems in an unbounded domain have received much attention in both engineering and mathematical communities over the years since they are of great importance in diverse areas of applications such as geophysics, seismology and non-destructive testing. These problems present considerable mathematical and computational challenges such as the oscillating character of solutions and the unbounded domain to be considered. In this paper, we consider the elastic wave scattering of a time harmonic incident wave by an impenetrable and bounded obstacle in two dimensions. The elastic scattered field can be modeled by a time-harmonic Navier equation \textcolor{rot}{together with} an appropriate radiation condition at infinity and a boundary condition on the boundary of the obstacle.

Recently, various methods have been developed for the numerical solutions of elastic scattering problems. Among them, the boundary integral equation (BIE) method (\cite{HW08}), which has also been widely used in acoustics, electromagnetics and elastostatics (\cite{BET12,HW04,HX11,N01,N02}), takes some advantages over domain discretization methods since it is not necessary for the BIE method to impose any artificial boundary condition for the radiation condition, and the reduced \textcolor{rot}{BIEs} are only discretized on the boundary of the obstacle. To effectively reduce the \textcolor{rot}{BIE} into a linear system, many different solvers including the boundary element method (BEM) (\cite{MB88}), the Nystr\"ome method (\cite{TC07}), the fast multipole method \cite{CBS08,TC09} and the spectral method (\cite{L14}) have been considered.

In this paper, we are concerned with the Galerkin BEM for solving the two dimensional elastic scattering problem with Neumann boundary condition. There are many advantages for the application of Galerkin schemes to \textcolor{rot}{BIEs}, including the availability of full mathematical convergence analysis allowing $h$-$p$  approximations,  the optimality of the method. We first introduce a Burton-Miller BIE formulation (\cite{BM71})  solving the original boundary value problem. Indeed, this type of BIE inherits from the original problem the uniqueness of solution for all frequencies regardless of  types of boundary conditions or types of incident waves. Then, the corresponding variational equation is shown to admit a unique solution using the Fredholm's alternative. What we have to pay for the achievement of  uniqueness  is that the computational formulation consists of all four boundary integral operators corresponding to the time-harmonic Navier equation, including the singular and hyper-singular boundary integral operators.  As a result, how these classically non-integrable boundary integral operators are evaluated accurately and effectively is crucial for the treatment of elastic waves using BIE method. A semi-classical method based on local polar coordinates and a Laurent series expansion of the relevant integrand is applied in \cite{BLR14} to evaluate the hyper-singular integral in the sense of Cauchy principle value and Hadamard finite part sense. Another  idea to evaluate hyper-singular integrals is to reduce the order of singularity of hyper-singular boundary integrals. For this purpose, with the help of subtraction and addition of relevant terms, the hyper-singular boundary integral could be reduced to a form involving at most weakly singular integrals (\cite{LR93}). In this work, by utilizing the duality paring in weak forms and the tangential G\"unter derivative (\cite{HW08,KGBB79}),  we present a new and accurate regularization formula (\cite{YHX}) which replaces the weak form of  \textcolor{rot}{hyper-singular} boundary integral operator  by a coupling of several weakly-singular boundary integrals. As applying the Galerkin scheme to the variation equation, we simply use piecewise linear or constant basis functions, and linear approximations of integral curves. During the  computation of entries of the coefficient matrix in the reduced linear system, we propose a novel strategy based on the series expansions of Hankel functions and several special integrals to evaluate all weakly-singular integrals exactly. In addition, all non-singular integrals can be approximated using Gaussian quadrature rules, and all singular integrals are vanishing due to the fact of using line integral curves.

The rest of this paper is organized as follows. In Section 2, a Burton-Miller BIE formulation for the considered exterior elastic scattering problem is presented and the solvability of the corresponding variational equation is studied. In Section 3, we discuss the numerical procedures of the Galerkin scheme for approximating the variational equation and present a new strategy to evaluate all weakly-singular boundary integrals exactly through using series representations of special Hankel functions. Several numerical examples are presented in Section 4 to verify the efficiency and accuracy of the proposed numerical methods. The paper is concluded in Section 5 with some general conclusions and remarks for future work.

\section{Mathematical problems}
Let $\Omega\subset \R^2$ be a bounded, simply connected and impenetrable body with sufficiently smooth boundary $\Gamma=\partial\Omega$, and its exterior complement is denoted by $\Omega^c= \R^2\setminus\overline{\Omega}\subset \R^2$. The domain $\Omega^c$ is occupied by a linear and isotropic elastic solid determined through the Lam\'e constants $\lambda$ and $\mu$ ($\mu>0$, $\lambda+\mu>0$) and its mass density $\rho>0$. \textcolor{rot}{Denote by $\omega>0$} the frequency of propagating elastic waves. The problem to be considered is to determine the elastic displacement field ${\bb u}$ in the solid provided an incident field ${\bb u}^i$, and can be formulated as follows: Given ${\bb u}^i$, find ${\bb u}=(u_1,u_2)^\top\in (C^2(\Omega^c)\cap C^1(\overline{\Omega^c}))^2$ satisfying
\be
\label{Navier}
\Delta^{*}{\bb u} +   \rho \omega^2{\bb u} &=& {\bb 0}\quad\mbox{in}\quad \Omega^c,\\
\label{BoundCond}
{\bb T}({\bb u}+{\bb u}^i) &=& {\bb 0} \quad \mbox{on}\quad \Gamma,
\en
 and the Kupradze radiation condition (\cite{KGBB79})
\be
\label{RadiationCond}
\lim_{r \to \infty} r^{1/2}\left(\frac{\partial {\bf u}_t}{\partial r}-ik_t{\bf u}_t\right) = 0,\quad r=|x|,\quad t=p,s,
\en
 uniformly with respect to all $\hat{x}=x/|x|\in\mathbb{S}:=\{x=(x_1,x_2)^\top\in\R^2:|x|=1\}$.  Here, ${\bb u}={\bb u}_p+{\bb u}_s$, and the compressional wave ${\bb u}_p$ and the shear wave ${\bb u}_s$ are given by
\ben
{\bb u}_p=-\frac{1}{k_p^2}\mbox{grad}\,\mbox{div}\,{\bb u},\quad {\bb u}_s=\frac{1}{k_s^2}\overrightarrow{\mbox{curl}}\,\mbox{curl}\,{\bb u}
\enn
with
\ben
k_s = \omega\sqrt{\frac{\rho}{\mu}},\quad k_p = \omega \sqrt{\frac{\rho}{\lambda + 2\mu}},
\enn
and
\ben
\overrightarrow{\mbox{curl}}=\left(\frac{\pa}{\pa x_2}, -\frac{\pa}{\pa x_1}\right)^\top,\quad \mbox{curl}\,{\bb v}=\frac{\pa v_2}{\pa x_1} -\frac{\pa v_1}{\pa x_2},\quad {\bb{v}}=( v_1,  v_2)^\top.
\enn
In addition, $\Delta^{*}$ is the operator defined by
\be
\label{LameOper}
\Delta^* = \mu\,\mbox{div}\,\mbox{grad} + (\lambda + \mu)\,\mbox{grad}\, \mbox{div}\,,
\en
and ${\bb T}$ is the traction operator on the boundary given by
\be
\label{TracationOper}
{\bb T}{\bb u} = 2\mu\frac{\pa\bb u}{\pa\bb n} + \lambda (\mbox{div}\,{\bb u}){\bb n} +\mu\,{\bb n}\times\mbox{curl}\,{\bb u},
\en
\textcolor{rot}{where ${\bb n}$} is the outward unit normal to the boundary $\G$. For the uniqueness of the classical boundary value problem (\ref{Navier})-(\ref{RadiationCond}), we refer to \cite{KGBB79}.



\subsection{Boundary integral equations}
\label{sec:3}
It follows from Green's representation formula (\cite{HW08}) that the unknown functions \textcolor{rot}{${\bb u}$} can be represented in the form
\be
\label{DirectBRF}
{\bb u}(x)  =  \int_{\Gamma}({\bb T}_y{\bb E}(x,y))^\top{\bb u}(y)\,ds_y - \int_{\Gamma}{\bb E}(x,y){\bb t}(y)\,ds_y , \quad \forall\,x\in\Omega^c,
\en
where ${\bb t}=(\bb{Tu})|_\G$ and ${\bb E}(x,y)$ is the fundamental displacement tensor of the time-harmonic Navier equation (\ref{Navier}) taking the form
\be
\label{NavierFS}
{\bb E}(x,y) = \frac{1}{\mu}\gamma_{k_s}(x,y){\bb I} + \frac{1}{\rho\omega^2}\nabla_x\nabla_x\left[\gamma_{k_s}(x,y) - \gamma_{k_p}(x,y)\right],\quad x\ne y.
\en
In (\ref{NavierFS}) and the following, ${\bb I}$ denotes the identity matrix, and $\gamma_{k_t}(x,y)$ is the fundamental solution of the Helmholtz equation in $\R^2$ with wave number $k_t$\textcolor{rot}{ , i.e.,}
\be
\label{HelmholtzFS}
\gamma_{k_t}(x,y) = \frac{i}{4}H_0^{(1)}(k_t|x-y|),\quad x\ne y,\quad t=p,s,
\en
where $H_0^{(1)}(\cdot)$ is the first kind  Hankel function of order zero. Now letting  $x$ in equations  (\ref{DirectBRF}) approach to the boundary $\G$ and applying the jump conditions, we obtain the corresponding \textcolor{rot}{BIE} on $\G$
\be
\label{ElasticBIE1}
{\bb u}(x) = \left(\frac{1}{2}I + K\right){\bb u}(x)-V{\bb t}(x) ,\quad x\in\Gamma.
\en
Operating with the traction operator on (\ref{DirectBRF}), taking the limits as $x\to\G$ and applying the jump relations, we are led to the second \textcolor{rot}{BIE} on $\G$
\be
\label{ElasticBIE2}
{\bb t}(x) = \left(\frac{1}{2}I - K'\right){\bb t}(x) - W{\bb u}(x),\quad x\in\Gamma.
\en
In the \textcolor{rot}{BIEs} (\ref{ElasticBIE1})-(\ref{ElasticBIE2}), $I$ is the identity operator and the boundary integral operators for the elasticity are defined by
\be
\label{SolidBIO1}
V{\bb t}(x) &=& \int_{\Gamma}{\bb E}(x,y) {\bb t}(y)\,ds_y,\quad x\in\Gamma,\\
\label{SolidBIO2}
K{\bb u}(x) &=& \int_{\Gamma}({\bb T}_y{\bf E}(x,y))^\top{\bb u}(y)\,ds_y,\quad x\in\Gamma,\\
\label{SolidBIO3}
K'{\bb t}(x) &=& \int_{\Gamma}{\bb T}_x{\bf E}(x,y){\bb t}(y)\,ds_y,\quad x\in\Gamma,\\
\label{SolidBIO4}
W{\bb u}(x)  &=& -{\bb T}_x\int_{\Gamma}({\bb T}_y{\bb E}(x,y))^\top{\bb u}(y)\,ds_y,\quad x\in\Gamma.
\en
Here, $V$, $K$, $K'$ and $W$ denote the single-layer, double-layer, transpose of double-layer and hyper-singular boundary integral operators, respectively. By combining the \textcolor{rot}{BIEs} (\ref{ElasticBIE1})-(\ref{ElasticBIE2}), we obtain the so-called Burton-Miller formulation (\cite{BM71})
\be
\label{ElasticBIE3}
\left[W+i\eta\left(\frac{1}{2}I - K\right)\right]{\bb u}(x)+\left[\frac{1}{2}I + K'+i\eta V\right]{\bb t}(x)=0,\quad x\in\Gamma,
\en
where $\eta\ne 0$ is called the combination coefficient. Using the boundary condition \eqref{BoundCond},  the above formulation leads to
\be
\label{NBIE}
\left[W+i\eta\left(\frac{1}{2}I - K\right)\right]{\bb u}(x)=\left[\frac{1}{2}I + K'+i\eta V\right]({\bb T}{\bb u}^i)(x) =:{\bb f},\quad x\in\Gamma.
\en

\begin{theorem}
The boundary integral equation (\ref{NBIE}) is uniquely solvable.
\end{theorem}
\begin{proof}
It is sufficient to prove that the corresponding homogeneous equation of (\ref{NBIE}) has only the trivial solution. Suppose that \textcolor{rot}{${\bb u}_0$} is a solution of the corresponding
homogeneous equation of (\ref{NBIE}). Let
\ben
{\bb u}_i(x)  =  \int_{\Gamma}({\bb T}_y{\bb E}(x,y))^\top{\bb u}_0(y)\,ds_y, \quad \forall\,x\in\Omega.
\enn
Then we obtain from the homogeneous form of (\ref{NBIE}) that
\ben
{\bb Tu}_i-i\eta{\bb u}_i=0 \quad\mbox{on}\quad\Gamma.
\enn
Applying Betti's formula to ${\bb u}_i$ and its complex conjugate $\ov{{\bb u}_i}$ we obtain
\ben
0 &=& \int_{\Omega}\left({\bb u}_i\cdot\Delta^*\ov{{\bb u}_i}-\ov{{\bb u}_i}\cdot\Delta^*{\bb u}_i\right)dx\\
&=& \int_{\Gamma}\left({\bb u}_i\cdot{\bb T}\ov{{\bb u}_i}-\ov{{\bb u}_i}\cdot{\bb T}{\bb u}_i\right)ds\\
&=& -2i\eta\int_{\Gamma}|{\bb u}_i|^2\,ds.
\enn
Then it follows that ${\bb u}_i=0$ on $\Gamma$ and therefore, ${\bb Tu}_i=0$ on $\Gamma$. On the other hand, let
\ben
{\bb u}_e(x)  =  \int_{\Gamma}({\bb T}_y{\bb E}(x,y))^\top{\bb u}_0(y)\,ds_y, \quad \forall\,x\in\Omega^c.
\enn
It holds that on $\Gamma$
\ben
{\bb Tu}_e={\bb Tu}_i=0.
\enn
Then the uniqueness for the exterior elastic scattering problem implies that ${\bb u}_e=0$ in $\Omega^c$ and hence ${\bb u}_e=0$ on $\Gamma$. Evaluating the jump on the boundary $\Gamma$, we have
\ben
{\bb u}_0={\bb u}_e|_{\Gamma}-{\bb u}_i|_{\Gamma}=0 .
\enn
The proof now is complete.
\end{proof}

\subsection{Weak formulation}
The standard weak formulation of (\ref{NBIE}) reads:  Given ${\bb T}{\bb u}^i\in(H^{-1/2}(\G))^2$, find ${\bb u}\in (H^{1/2}(\G))^2$ such that
\be
\label{weak}
A({\bb u},{\bb v})=F({\bb v}) \quad\mbox{for all}\quad {\bb v}\in(H^{1/2}(\G))^2,
\en
where the sesquilinear form $A(\cdot\,,\,\cdot):(H^{1/2}(\G))^2\times (H^{1/2}(\G))^2\mapsto \C$ is defined by
\be
\label{sesquilinear}
A({\bb u},{\bb v})=\left\langle \left[W+i\eta\left(\frac{1}{2}I - K\right)\right]{\bb u}, {\bb v} \right\rangle,
\en
and the linear functional $F({\bb v})$ on $(H^{1/2}(\G))^2$ is defined by
\ben
F({\bb v})=\left\langle {\bb f},{\bb v} \right\rangle.
\enn
Here, $\left\langle \cdot,\cdot \right\rangle$ is the standard $L^2$ duality pairing between $(H^{-1/2}(\G))^2$ and $(H^{1/2}(\G))^2$.

\begin{theorem}
The sesquilinear form (\ref{sesquilinear}) satisfies a G\aa rding's inequality in the form
\be
\label{garding}
\mbox{Re}\,\{A({\bb u},{\bb u})\}\ge \alpha\|{\bb u}\|_{(H^{1/2}(\Gamma))^2}-\beta\|{\bb u}\|_{(H^{1/2-\epsilon}(\Gamma))^2},
\en
for all ${\bb u}\in (H^{1/2}(\G))^2$. Here, $\mbox{Re}\{\,\}$ implies the real part, and  $\alpha>0$, $\beta\ge 0$ and $0<\epsilon<1/2$ are all constants.
\end{theorem}
\begin{proof}
From the estimates in \cite{HKR00} we know that
\be
\label{Garding1}
\mbox{Re}\,\{\langle W{\bf u}, {\bf u}\rangle\} \ge \alpha\|{\bf u}\|_{(H^{1/2}(\Gamma))^2}^2 - \beta\|{\bf u}\|_{(H^{1/2-\epsilon}(\Gamma))^2}^2
\en
for some constants $\alpha> 0$, $\beta\ge 0$ and $0<\epsilon<1/2$. On the other hand, we have
\ben
\left|\left\langle i\eta\left(\frac{1}{2}I - K\right){\bb u}, {\bb u} \right\rangle\right| &\le& |\eta|\left\|\left(\frac{1}{2}I - K\right){\bb u}\right\|_{(H^0(\Gamma))^2} \|{\bb u}\|_{(H^0(\Gamma))^2}\\
&\le& c\|{\bb u}\|_{(H^0(\Gamma))^2}^2\\
&\le& c\|{\bb u}\|_{(H^{1/2-\epsilon}(\Gamma))^2}^2,
\enn
which implies that
\be
\label{Garding2}
\mbox{Re}\,\left\{\left\langle i\eta\left(\frac{1}{2}I - K\right){\bb u}, {\bb u} \right\rangle\right\}\ge -c\|{\bb u}\|_{(H^{1/2-\epsilon}(\Gamma))^2},
\en
where $c>0$ and $0<\epsilon<1/2$ are all constants. Therefore, the combination of inequalities (\ref{Garding1}) and (\ref{Garding2}) gives the G\aa rding's inequality (\ref{garding}) immediately, and this completes the proof.
\end{proof}

Now, the existence result follows immediately from the Fredholm's Alternative: uniqueness implies existence. Therefore, we have the following theorem.

\begin{theorem}
The variational equation (\ref{weak}) admits a unique solution.
\end{theorem}

\section{Numerical schemes}
In this section, we discuss numerical procedures for approximating the variational equation \eqref{weak}. First, we present a Galerkin equation corresponding to  \eqref{weak}, and then arrive at a discrete linear system through using \textcolor{rot}{linear basis as} test functions, and linear approximations of integral curves. Two new techniques are to be utilized  during the discretization of the Galerkin equation. One is to apply a new regularization formula (\cite{YHX}) to the \textcolor{rot}{hyper-singular} boundary integral operator \eqref{SolidBIO4}, and as a result, only weakly-singular terms are remained in its practical computational formulations. The other is to compute all weakly-singular boundary integrals through using series representations of special Hankel functions. With the help of  the second technique, together with the fact of using line integral curves, all weakly-singular boundary integrals could be evaluated exactly, and all singular boundary integrals are vanishing. We point out that  low-order basis functions are not preconditions for the employment of these two techniques, and actually, coupled with these two techniques,  basis functions of any order  could be adopted in simulations to realize the \textcolor{rot}{$p$-version} of boundary element methods.

\subsection{Boundary element methods}

Let $\mathcal{H}_h$ be a finite dimensional subspace of $(H^{1/2}(\G))^2$. We consider the following problem: Given ${\bb T}{\bb u}^i$, find ${\bb u}_h\in \mathcal{H}_h$ such that
\be
\label{Galerkin}
A({\bb u}_h,{\bb v}_h)=F({\bb v}_h) \quad\mbox{for all}\quad {\bb v}_h\in\mathcal{H}_h.
\en
In particular, (\ref{Galerkin}) is known as the Galerkin approximation of (\ref{weak}). We refer to \cite{HW04} for fundamental  features of \eqref{Galerkin}, including the well-posedness and the numerical error bounds. In this work,  we only describe  a brief procedure of reducing the Galerkin equation (\ref{Galerkin}) to its discrete linear system of equations.
Let $x_i$, $i=1,2,...,N$ be  discretion points on $\G$,  and $\G_i$ be the line segment between $x_i$ and $x_{i+1}$. Here, we set
\ben
\G_{N+1}=\G_1,\quad \G_{-1}=\G_{N},\quad x_{N+1}=x_1,\quad x_{-1}=x_{N}.
\enn
The outward unit normal and tangential to the boundary $\G_i$ are given respectively by
\ben
{\bf n}_{\G_i}=-{\bf N}\frac{x_{i+1}-x_i}{|\G_i|}, \quad {\bf t}_{\G_i}=\frac{x_{i+1}-x_i}{|\G_i|},\quad|\G_i|=|x_{i+1}-x_i|,\quad {\bf N}=\begin{bmatrix}
0&{ - 1}\\
1&0
\end{bmatrix}.
\enn
Then the boundary $\G$ is approximated by
\ben
\tilde{\G}:=\bigcup_{i=1}^N\G_i.
\enn
For $x\in\G_i, i=1,...,N$, we introduce
\ben
x=x(\xi)=x_i+\frac{1+\xi}{2}(x_{i+1}-x_i),\quad \xi\in[-1,1].
\enn
Let $\{\varphi_i\}, i=1,2,...,N$ be piecewise linear basis functions of $\mathcal{H}_h$. We choose them as for $i=1,2,...,N$,
\ben
\varphi_i(x)=\varphi_i(x(\xi))=
\begin{cases}
\frac{1+\xi}{2}, & x(\xi)\in\G_{i-1},
\cr \frac{1-\xi}{2} , & x(\xi)\in\G_i,
\cr 0 , & \mbox{otherwise}.
\end{cases}
\enn
We seek the approximate solution \textcolor{rot}{${\bb u}_h$} in the forms
\ben
\textcolor{rot}{{\bb u}_h}(x)=\sum_{i=1}^N {\bf u}_i\varphi_i(x),
\enn
where ${\bf u}_i\in\C^2$, $i=1,...,N$ are unknown nodal values of \textcolor{rot}{${\bb u}_h$} at $x_i$. The given Cauchy data \textcolor{rot}{${\bb T}{\bb u}^i$} is interpolated with the form
\ben
\textcolor{rot}{{\bb T}{\bb u}^i(x)}=\sum_{i=1}^N {\bf g}_i\psi_i(x),
\enn
where ${\bf g}_i$, $i=1,...,N$ are known nodal values of \textcolor{rot}{${\bb Tu}^i$} at $x_i$, and $\{\psi_i\}, i=1,2,...,N$ are piecewise constant basis functions defined as
\ben
\psi_i(x)=\psi_i(x(\xi))=
\begin{cases}
1 , & x(\xi)\in\G_i,
\cr 0 , & \mbox{otherwise}.
\end{cases}
\enn
Substituting these interpolation forms into (\ref{Galerkin}) and setting $\varphi_i, i=1,2,...,N$ as test functions, we arrive at a linear system of equations
\be
\label{linearsys}
{\bf A}_h {\bf X}={\bf B}_h{\bf b},\quad{\bf A}_h\in\C^{2N\times 2N},\quad{\bf B}_h\in\C^{2N\times 1},
\en
where
\ben
{\bf A}_h=-\frac{1}{2}{\bf I}_{1h}+{\bf K}_{h}+\eta{\bf W}_{h},\quad
{\bf B}_h=-\left[{\bf V}_{h}-\eta\left( \frac{1}{2}{\bf I}_{2h}+ {\bf K}'_{h} \right)\right],
\enn
and
\ben
{\bf X}  = ({\bf u}_1^\top,{\bf u}_2^\top,...,{\bf u}_N^\top)^\top, \quad
{\bf b} = ({\bf g}_1^\top,{\bf g}_2^\top,...,{\bf g}_N^\top)^\top.
\enn
The entries ($\in\C^{2\times 2}$) of the corresponding matrixes are defined by
\be
\label{entry1}
{\bf V}_{h}(i,j) &=& \int_{\tilde{\G}} (V\psi_j)\varphi_i\,ds, \quad
{\bf W}_{h}(i,j) = \int_{\tilde{\G}} (W\varphi_j)\varphi_i\,ds,\\
\label{entry2}
{\bf K}_{h}(i,j) &=& \int_{\tilde{\G}} (K\varphi_j)\varphi_i\,ds, \quad
{\bf K}'_{h}(i,j) = \int_{\tilde{\G}} (K'\psi_j)\varphi_i\,ds,\\
\label{entry3}
{\bf I}_{1h}(i,j) &=& \int_{\tilde{\G}} \varphi_j\varphi_i\,ds\,\textcolor{rot}{\bb I}, \quad
{\bf I}_{2h}(i,j) = \int_{\tilde{\G}} \psi_j\varphi_i\,ds\,\textcolor{rot}{\bb I}.
\en

\subsection{Regularized formulations}

It can be seen that the \textcolor{rot}{BIE} (\ref{NBIE}) consists of both singular and \textcolor{rot}{hyper-singular} boundary integral operators, and extra treatments are needed for its numerical simulations. Thanks to the variational form and the tangential G\"unter derivative, a new regularization \textcolor{rot}{formula} for the \textcolor{rot}{hyper-singular} boundary integral operator has recently been derived  in \cite{YHX}, and takes the form
\be
\label{Ws}
\langle W{\bf w},{\bf v}\rangle &=&\mu k_{s}^2\int_{\Gamma}\int_{\Gamma}{\left(\left[{\bf n}_{x}{\bf n}_{y}^\top R-{\bf n}_{x}^\top{\bf n}_{y}\textcolor{rot}{\bb I}\gamma _{k_{s}}+{\bf N}{\bf n}_{x}^\top {\bf t}_{y}\gamma _{k_{s}}\right] {\bf w}(y)\right)^\top\ov{\bf v}(x)ds_{y}ds_x }\nonumber \\
&+&4\mu ^{2}\int_{\Gamma}\int_{\Gamma} {\left(\textcolor{rot}{\bb E}(x,y)\frac{{d{\bf w}(y)}}{{ds_{y}} }\right)^\top\frac{d\ov{\bf v}(x)}{ds_x}ds_{y}ds_x}\nonumber \\
&-&\frac{{4\mu ^{2}}}{{\lambda + 2\mu} }\int_{\Gamma}\int_{\Gamma} {\gamma_{k_p}(x,y)\frac{{d{\bf w}(y)^\top}}{{ds_{y}} }\frac{d\ov{\bf v}(x)}{ds_x}ds_{y}ds_x} \nonumber\\
&+&2\mu \int_{\Gamma}\int_{\Gamma}\cdot  {\left({\bf n}_{x} \nabla _{x}^\top R(x,y){\bf N}\frac{{d{\bf w}(y)}}{{ds_{y}} }\right)^\top\ov{\bf v}(x)ds_{y}ds_x}\nonumber\\
&-& 2\mu\int_{\Gamma}\int_{\Gamma}{\left({\bf N}\nabla_{y} R( x,y){\bf n}_{y}^\top{\bf w}(y)\right)^\top\frac{d\ov{\bf v}(x)}{ds_x}ds_{y}ds_x},
\en
where
\ben
R(x,y)=\gamma_{k_s}(x,y)-\gamma_{k_p}(x,y).
\enn
Clearly, only weakly-singular kernels are involved in the \textcolor{rot}{regularization formula} (\ref{Ws}). On the other hand, according to Theorem A.3 in \cite{YHX}, we know that
\ben
\textcolor{rot}{{\bb T}_x{\bb E}}(x,y) &=& - {\bf n}_x \nabla_x^\top R(x,y) + \frac{\partial\gamma_{k_s}(x,y)}{\partial n_x}\textcolor{rot}{\bb I}+ {\bf N}\frac{d\left[2\mu\textcolor{rot}{\bb E}(x,y) - \gamma_{k_s}(x,y)\textcolor{rot}{\bb I}\right]}{ds_x},\\
\textcolor{rot}{{\bb T}_y{\bb E}}(x,y) &=& - {\bf n}_y \nabla_y^\top R(x,y) + \frac{\partial\gamma_{k_s}(x,y)}{\partial n_y}\textcolor{rot}{\bb I}+ {\bf N}\frac{d\left[2\mu\textcolor{rot}{\bb E}(x,y) - \gamma_{k_s}(x,y)\textcolor{rot}{\bb I}\right]}{ds_y}.
\enn
Then integration by parts implies that
\be
\label{Ks}
\left\langle K{\bf w},{\bf v} \right\rangle &=& \int_{\Gamma} \int_{\Gamma}  {\frac{{\partial \gamma _{k_{s}}(x,y) } }{{\partial n_{y}
}}{\bf w}(y)^\top\ov{\bf v}(x)ds_{y}ds_x}  \nonumber\\
&-& \int_{\Gamma}\int_{\Gamma}  {\left(\nabla_{y} R( x,y){\bf n}_{y}^\top {\bf w}(y)\right)^\top\ov{\bf v}(x)ds_{y}ds_x} \nonumber\\
&+&\int_{\Gamma}\int_{\Gamma}  {\left(\left[2\mu \textcolor{rot}{\bb E}( x,y) - \gamma _{k_{s}}(x,y)\textcolor{rot}{\bb I}\right]{\bf N}\frac{d{\bf w}(y)}{ds_y}\right)^\top \ov{\bf v}(x)ds_{y}ds_x}.
\en
and
\be
\label{Ksp}
\left\langle K'{\bf w},{\bf v} \right\rangle &=& \int_{\Gamma} \int_{\Gamma}{\frac{{\partial \gamma _{k_{s}}(x,y) } }{{\partial n_x
}}{\bf w}(y)^\top\ov{\bf v}(x)ds_{y}ds_x} \nonumber\\
&-& \int_{\Gamma}\int_{\Gamma}  {\left({\bf n}_x\nabla_x^\top R( x,y) {\bf w}(y)\right)^\top\ov{\bf v}(x)ds_{y}ds_x} \nonumber\\
&+&\int_{\Gamma}\int_{\Gamma}  {\left({\bf N}\left[2\mu \textcolor{rot}{\bb E}( x,y) - \gamma _{k_{s}}(x,y)\textcolor{rot}{\bb I}\right]{\bf w}(y)\right)^\top\frac{d\ov{\bf v}(x)}{ds_x}ds_yds_x}.
\en
Observe that the first terms  in  (\ref{Ks}) and (\ref{Ksp}), consisting of singular kernels, are consistent with the weak forms of double-layer and transpose of double-layer boundary integral operators associated with Helmholtz equations (\cite{HX11}), respectively, and the remaining terms are all weakly-singular.

\subsection{Computational formulations}
Recall that
\ben
\textcolor{rot}{\bb E}(x,y) = \frac{1}{\mu}\gamma_{k_s}(x,y)\textcolor{rot}{\bb I} + \frac{1}{\rho\omega^2}\nabla_x\nabla_x R(x,y),\quad x\ne y.
\enn
A direct calculation gives
\be
\label{RepE}
\textcolor{rot}{\bb E}(x,y)&=&  \frac{i}{4\mu}H_0^{(1)}(k_s|x-y|)\textcolor{rot}{\bb I}\nonumber\\
&-& \frac{i}{4\rho\omega^2|x-y|} \left[k_sH_1^{(1)}(k_s|x-y|)-k_pH_1^{(1)}(k_p|x-y|)\right]\textcolor{rot}{\bb I} \nonumber\\
&+& \frac{i(x-y)(x-y)^\top}{4\rho\omega^2|x-y|^2} \left[k_s^2H_2^{(1)}(k_s|x-y|)-k_p^2H_2^{(1)}(k_p|x-y|)\right] ,\quad x\ne y.
\en
From the series representation of   bessel functions $J_n(\cdot)$ and $Y_n(\cdot)$ in \cite{AS72}, we arrive at the following Lemma.
\begin{lemma}
\label{series}
For $x\ne y$, we have the following representation
\be
\label{rep1}
&\quad& k_sH_1^{(1)}(k_s|x-y|)-k_pH_1^{(1)}(k_p|x-y|) \nonumber\\
&=& \sum_{m=0}^\infty \left( C_m^1|x-y|^{2m+1} +C_m^2|x-y|^{2m+1}\ln|x-y|\right),\\
\label{rep2}
&\quad &k_s^2H_2^{(1)}(k_s|x-y|)-k_p^2H_2^{(1)}(k_p|x-y|)\nonumber\\
&=& \sum_{m=0}^\infty \left( C_m^3|x-y|^{2m+2} +C_m^4|x-y|^{2m+2}\ln|x-y|\right)-\frac{i(k_s^2-k_p^2)}{\pi},\\
\label{rep3}
&\quad& H_0^{(1)}(k|x-y|) \nonumber\\
&=& \sum_{m=0}^\infty \left[ \left(C_m^5+C_m^6\ln\frac{k}{2}\right) k^{2m}|x-y|^{2m} +C_m^6k^{2m}|x-y|^{2m}\ln|x-y|\right].
\en
Here,
\ben
C_m^1 &=&
\begin{cases}
\frac{k_s^2-k_p^2}{2}\left(1+\frac{2ic_e}{\pi}-\frac{i}{\pi}\right) +\frac{2i}{\pi}\left(k_s^2\ln\frac{k_s}{2}-k_p^2\ln\frac{k_p}{2}\right), &m=0 ,
\cr \left[ {\begin{array}{*{20}{c}}
\frac{(-1)^m\left(k_s^{2m+2}-k_p^{2m+2}\right)}{2^{2m+1}m!(m+1)!} \left[1+\frac{2ic_e}{\pi}-\frac{i}{\pi}\left( 2\sum_{l=1}^m\frac{1}{l} +\frac{1}{m+1}\right)\right]\\
+\frac{2i}{\pi}\left(k_s^{2m+2}\ln\frac{k_s}{2}-k_p^{2m+2}\ln\frac{k_p}{2}\right)
\end{array}} \right], &m\ge1,
\end{cases}\\
C_m^2 &=& \frac{2i(-1)^m}{\pi 2^{2m+1}m!(m+1)!}\left(k_s^{2m+2}-k_p^{2m+2}\right) ,\quad m\ge0,\\
C_m^3 &=&
\begin{cases}
\frac{k_s^4-k_p^4}{8}\left(1+\frac{2ic_e}{\pi}-\frac{3i}{2\pi}\right) +\frac{2i}{\pi}\left(k_s^4\ln\frac{k_s}{2}-k_p^4\ln\frac{k_p}{2}\right), &m=0 ,
\cr \left[ {\begin{array}{*{20}{c}}
\frac{(-1)^m\left(k_s^{2m+4}-k_p^{2m+4}\right)}{2^{2m+2}m!(m+2)!} \left[1+\frac{2ic_e}{\pi}-\frac{i}{\pi}\left( 2\sum_{l=1}^m\frac{1}{l} +\frac{1}{m+2}\right)\right] \\
+\frac{2i}{\pi}\left(k_s^{2m+4}\ln\frac{k_s}{2}-k_p^{2m+4}\ln\frac{k_p}{2}\right)
\end{array}} \right], &m\ge1,
\end{cases}\\
C_m^4 &=& \frac{2i(-1)^m}{\pi 2^{2m+3}m!(m+3)!}\left(k_s^{2m+4}-k_p^{2m+4}\right) ,\quad m\ge0,\\
C_m^5 &=&
\begin{cases}
1+\frac{2ic_e}{\pi}, &m=0 ,
\cr \frac{(-1)^m}{2^{2m}m!m!} \left[1+\frac{2ic_e}{\pi}-\frac{2i}{\pi}\sum_{l=1}^m\frac{1}{l}\right], &m\ge1,
\end{cases}\\
C_m^6 &=& \frac{2i(-1)^m}{\pi 2^{2m}m!m!},\quad m\ge0,
\enn
with $c_e$ being the Euler constant.
\end{lemma}

Next, we introduce some useful integrals for $m\ge 0$:
\ben
I_m^1 &=& \int_{-1}^1 \int_{-1}^1(\xi_1-\xi_2)^{2m+1}\xi_1\,d\xi_1d\xi_2 =\sum_{l=0}^{2m+1}\frac{(-1)^{l+1}C_{2m+1}^l\left[1-(-1)^l\right]^2}{(l+2)(2m+2-l)},
\enn
\ben
I_m^2 &=& \int_{-1}^1 \int_{-1}^1(\xi_1-\xi_2)^{2m+1}\xi_1\ln|\xi_1-\xi_2|\,d\xi_1d\xi_2,\\
&=& \frac{1}{2(m+1)^2}\sum_{l=0}^{2m+2}\frac{C_{2m+2}^l\left[1-(-1)^l\right]}{l+2} +\frac{2^{2m+3}\ln2}{(m+2)(2m+3)}\\
&-&\frac{(6m^2+18m+13)2^{2m+3}}{(m+1)(2m+3)^2(m+2)},
\enn
\ben
I_m^3 &=& \int_{-1}^1 \int_{-1}^1(\xi_1-\xi_2)^{2m}\,d\xi_1d\xi_2 =\frac{2^{2m+2}}{(2m+1)(m+1)},
\enn
\ben
I_m^4 &=& \int_{-1}^1 \int_{-1}^1(\xi_1-\xi_2)^{2m}\ln|\xi_1-\xi_2|\,d\xi_1d\xi_2 \\ &=&\frac{2^{2m+2}\ln2}{(2m+1)(m+1)}-\frac{(4m+3) 2^{2m+1}}{(2m+1)^2(m+1)^2},
\enn
\ben
I_m^5 &=& \int_{-1}^1 \int_{-1}^1(\xi_1-\xi_2)^{2m}\xi_1\xi_2\,d\xi_1d\xi_2 =\sum_{l=0}^{2m}\frac{(-1)^lC_{2m}^l\left[1-(-1)^l\right]^2}{(l+2)(2m+2-l)},
\enn
\ben
I_m^6 &=& \int_{-1}^1 \int_{-1}^1(\xi_1-\xi_2)^{2m}\xi_1\xi_2\ln|\xi_1-\xi_2|\,d\xi_1d\xi_2,\\
&=& -\frac{m 2^{2m+2}\ln2}{(2m+1)(m+1)(m+2)} \\
&-& \frac{1}{(2m+1)(m+1)} \left[\frac{2^{2m+2}}{(m+2)^2}+\frac{2^{2m+1}}{m+1}-\frac{2^{2m+3}}{2m+3}\right]\\
&+& \frac{1}{(m+1)^2(2m+1)^2}\sum_{l=1, l\;is\;odd}^{2m+1}C_{2m+1}^l \frac{(2m+1)^2}{l+2}\\
&-& \frac{1}{(m+1)^2(2m+1)^2}\sum_{l=0, l\;is\;even}^{2m}C_{2m+1}^l \frac{(4m+3)}{l+3}.
\enn
Among them, $\textcolor{rot}{I_m^j}, j=2,4,6$ are weakly-singular integrals and can be evaluated exactly.

We are now ready to present the  formulations to \textcolor{rot}{compute} the \textcolor{rot}{matrixes} in (\ref{entry1})-(\ref{entry3}), and here we only consider the matrix ${\bf V}_{h}$. All computational formulations associated to other matrix in (\ref{entry1})-(\ref{entry3}) are listed in Appendix. Firstly, we introduce some notations. Let us denote $\G^1=\G_{i-1}$ or $\G_i$ and $\G^2=\G_{j-1}$ or $\G_j$,  and denote the two vertexes of $\G^1$ and $\G^2$ by  $x^1,x^2$ and $y^1, y^2$, respectively. For $x\in\G^1$, $y\in\G^2$, we set
\ben
x&=&x(\xi_1)=x^1+\frac{1+\xi_1}{2}(x^2-x^1),\quad \xi_1\in[-1,1],\\
y&=&y(\xi_2)=y^1+\frac{1+\xi_2}{2}(y^2-y^1),\quad \xi_2\in[-1,1],
\enn
and
\ben
\varphi_i(x)=\varphi_i(x(\xi_1))=\frac{1+k_1\xi_1}{2}\quad\mbox{on}\quad\G^1,\quad \varphi_j(y)=\varphi_j(y(\xi_2))=\frac{1+k_2\xi_2}{2}\quad\mbox{on}\quad\G^2,
\enn
where
\ben
(k_1,k_2)=
\begin{cases}
(1,1), & \G^1\times\G^2=\G_{i-1}\times\G_{j-1},
\cr  (-1,1), & \G^1\times\G^2=\G_i\times\G_{j-1},
\cr  (1,-1), & \G^1\times\G^2=\G_{i-1}\times\G_j,
\cr  (-1,-1), & \G^1\times\G^2=\G_i\times\G_j.
\end{cases}
\enn
We write
\ben
{\bf V}_{h}(i,j) &=& \int_{\tilde{\G}} (V\psi_j)\varphi_i\,ds\\
&=& \int_{\G_{i-1}\cup\G_i}\int_{\G_j} \textcolor{rot}{\bb E}(x,y)\varphi_i(x)\,ds_yds_x.
\enn
Then it suffices to evaluate the integral
\ben
V_0(\G^1,\G^2,k_1,i,j)=\int_{\G^1}\int_{\G^2} \textcolor{rot}{\bb E}(x,y)\varphi_i(x)\,ds_yds_x,
\enn
which further implies that
\ben
{\bf V}_{h}(i,j) = V_0(\G_{i-1},\G_j,1,i,j) +V_0(\G_i,\G_j,-1,i,j).
\enn
We obtain from (\ref{RepE}) that
\ben
\label{VSij1}
&\quad& V_0(\G^1,\G^2,k_1,i,j) \nonumber\\
&=& \frac{i}{4\mu}\int_{\G^1}\int_{\G^2} H_0^{(1)}(k_s|x-y|)\varphi_i(x)\textcolor{rot}{\bb I}\,ds_yds_x \nonumber\\
&-& \frac{i}{4\rho\omega^2} \int_{\G^1}\int_{\G^2} \frac{k_sH_1^{(1)}(k_s|x-y|)-k_pH_1^{(1)}(k_p|x-y|)}{|x-y|}\varphi_i(x)\textcolor{rot}{\bb I}\,ds_yds_x \nonumber\\
&+& \frac{i}{4\rho\omega^2} \int_{\G^1}\int_{\G^2} \frac{k_s^2H_2^{(1)}(k_s|x-y|)-k_p^2H_2^{(1)}(k_p|x-y|)}{|x-y|^2} (x-y)(x-y)^\top \varphi_i(x)ds_yds_x,
\enn
which further yields,  with ${\bf r}=x(\xi_1)-y(\xi_2)$,
\be
\label{VSij2}
&\quad& V_0(\G^1,\G^2,k_1,i,j) \nonumber\\
&=& \frac{i|\G^1||\G^2|}{32\mu}\int_{-1}^1\int_{-1}^1 H_0^{(1)}(k_s|{\bf r}|)(1+k_1\xi_1)\textcolor{rot}{\bb I}\,d\xi_1d\xi_2 \nonumber\\
&-& \frac{i|\G^1||\G^2|}{32\rho\omega^2} \int_{-1}^1\int_{-1}^1 \frac{k_sH_1^{(1)}(k_s|{\bf r}|)-k_pH_1^{(1)}(k_p|{\bf r}|)}{|{\bf r}|} (1+k_1\xi_1)\textcolor{rot}{\bb I}\,d\xi_1d\xi_2 \nonumber\\
&+& \frac{i|\G^1||\G^2|}{32\rho\omega^2} \int_{-1}^1\int_{-1}^1 \frac{k_s^2H_2^{(1)}(k_s|{\bf r}|)-k_p^2H_2^{(1)}(k_p|{\bf r}|)}{|{\bf r}|^2} (1+k_1\xi_1){\bf r}{\bf r}^\top d\xi_1d\xi_2.
\en
If $\G^1\ne\G^2$, the formula (\ref{VSij2}) can be \textcolor{rot}{approximated} directly by Gauss quadrature rules. If $\G^1=\G^2$, by Lemma \ref{series} we have
\be
\label{VSij3}
&\quad& V_0(\G^1,\G^2,k_1,i,j) \nonumber\\
&=& \sum_{m=0}^\infty \frac{ik_s^{2m}|\G^1|^{2m+2}}{2^{2m+5}\mu} \left[\left(C_m^5+C_m^6\ln\frac{k_s|\G^1|}{4}\right)I_m^3 +C_m^6I_m^4\right]\textcolor{rot}{\bb I} \nonumber\\
&-& \sum_{m=0}^\infty \frac{i|\G^1|^{2m+2}}{2^{2m+5}\rho\omega^2} \left[\left(C_m^1+C_m^2\ln\frac{|\G^1|}{2}\right)I_m^3 +C_m^2I_m^4\right] \textcolor{rot}{\bb I}\nonumber\\
&+& \sum_{m=0}^\infty \frac{i|\G^1|^{2m+4}}{2^{2m+7}\rho\omega^2} \left[\left(C_m^3+C_m^4\ln\frac{|\G^1|}{2}\right)I_{m+1}^3 +C_m^4I_{m+1}^4\right]{\bf t}_{\G_1}\,{\bf t}_{\G_1}^\top \nonumber\\
&+& \frac{|\G^1|^2(k_s^2-k_p^2)}{8\pi\rho\omega^2}{\bf t}_{\G_1}\,{\bf t}_{\G_1}^\top.
\en
Finally, we point out that  the infinite series should be truncated into finite ones in practical computing. Let $M$ be the truncation number of the series, that is, we only use the $M+1$ leading terms. Usually, $M=20$ are large enough for the achievement of optimal order of accuracy for the numerical tests to be presented in the next Section.

\section{Numerical examples}

In this section, we present several numerical tests to demonstrate the efficiency and accuracy of the proposed boundary element method for solving two dimensional elastic wave scattering problems. Unless otherwise stated, we always set $\eta=1$, $\rho=1$, $\lambda=2$, $\mu=1$ which implies that $k_s=2k_p$. We use a direct solver  for solutions of the  linear system (\ref{linearsys}). Impenetrable obstacles occupying the domain $\Omega$ with different boundary shapes are considered in our tests, which are listed in Figure \ref{Fig1}.
\begin{figure}[htbp]
\centering
\begin{tabular}{ccc}
\includegraphics[scale=0.25]{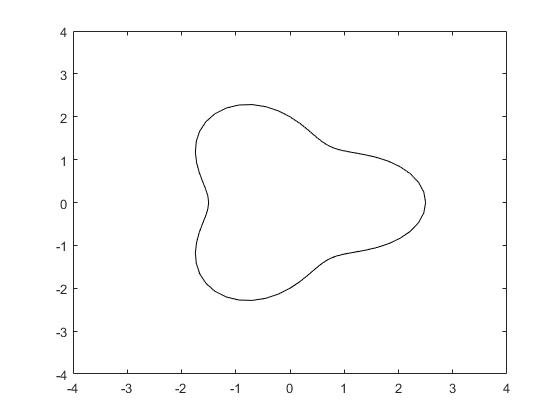} &
\includegraphics[scale=0.25]{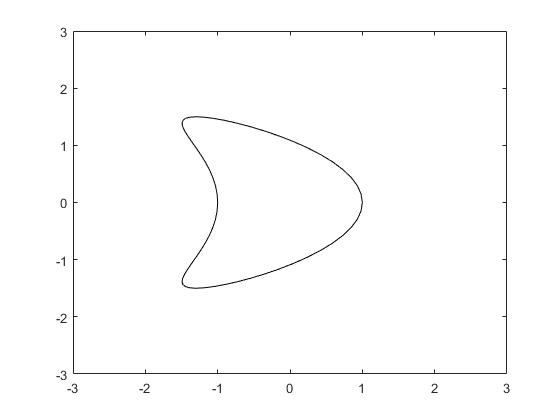} &
\includegraphics[scale=0.25]{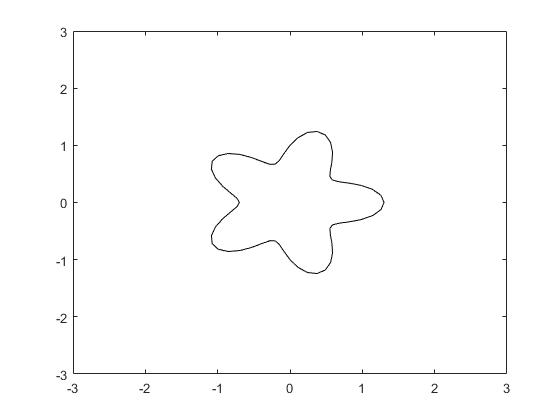} \\
(a) rounded-triangle-sharped & (b) kite-shaped & (c) star-shaped
\end{tabular}
\begin{tabular}{cc}
\includegraphics[scale=0.25]{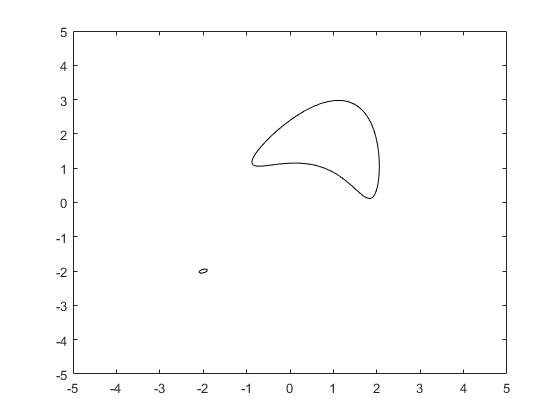} &
\includegraphics[scale=0.25]{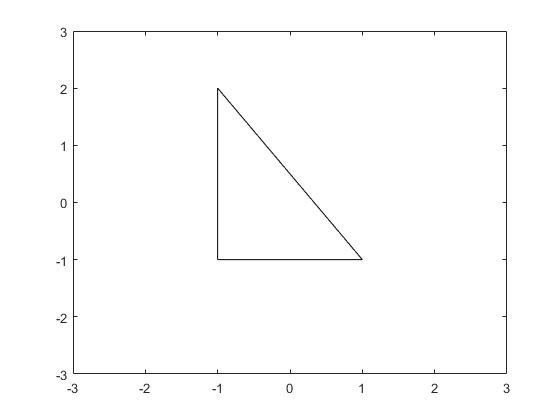} \\
(d) mixed-shaped& \textcolor{rot}{(e) right-angled-triangle-shaped}
\end{tabular}
\caption{Impenetrable obstacles to be considered in numerical tests.}
\label{Fig1}
\end{figure}

{\bf Example 1} ({\it accuracy of the series expansion of Hankel functions}). In the first example, we test the accuracy of the series expansions derived in Lemma 3.1. Denote
\ben
F_1&:=&k_sH_1^{(1)}(k_s|x-y|)-k_pH_1^{(1)}(k_p|x-y|),\\
F_2&:=&k_s^2H_2^{(1)}(k_s|x-y|)-k_p^2H_2^{(1)}(k_p|x-y|),\\
F_3&:=&H_0^{(1)}(k_s|x-y|).
\enn
Since we only use these expansions for computing weakly singular integrals, and it means that $|x-y|$ is relatively small. In this example, we set $|x-y|$ to be half of the shear wave length, that is, $|x-y|=\pi/k_s$. First, we test the numerical errors of $F_i, i=1,2,3$ with respect to the truncation number $M$ of the series. We choose $\omega=1$ and the absolute errors are presented in Figure \ref{Fig2}(a). It can be seen that we obtain highly accurate values of $F_i, i=1,2,3$ as $M\ge 15$. Next, we investigate  the numerical errors of $F_i, i=1,2,3$ as the frequency increasing. Now, we fix $M=20$ and choose different frequencies from $\pi$ to $21\pi$. Observe from Figure \ref{Fig2}(b) that the numerical errors of $F_1$, $F_2$ and $F_3$ are of order $O(\omega^2)$, $O(\omega)$ and $O(1)$, respectively. This result shows that the numerical errors of $F_i, i=1,2,3$ are controllable with respect to the change of frenquency.

\begin{figure}[htbp]
\centering
\begin{tabular}{cc}
\includegraphics[scale=0.35]{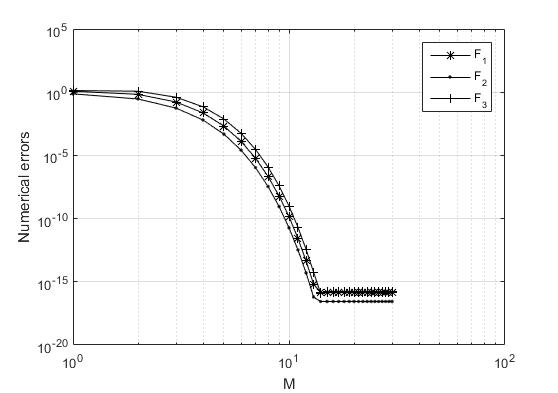} &
\includegraphics[scale=0.35]{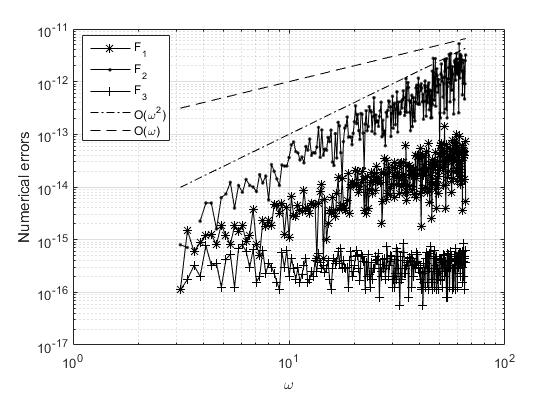} \\
(a) & (b)
\end{tabular}
\caption{Numerical errors of the series expansion for $F_i, i=1,2,3$ with respect to : (a) the truncation number $M$ ; (b) the frequency $\omega$.}
\label{Fig2}
\end{figure}

{\bf Example 2} ({\it accuracy of BEM with low frequency}). In this example, we consider the elastic wave solutions with low frequencies. Assume that the obstacle $\Om$ is rounded-triangle-shaped or kite-shaped,  and their boundaries are characterized by
\ben
x(t)=(2+0.5\cos3t)(\cos t,\sin t),\quad t\in[0,2\pi),
\enn
and
\ben
x(t)=(\cos t+0.65\cos2t-0.65,1.5\sin t),\quad t\in[0,2\pi),
\enn
respectively. We choose the incident wave such that the exact solution is
\ben
\textcolor{rot}{{\bb u}}=-\nabla\varphi,\quad \varphi(x)=H_0^{(1)}(k_p|x|),\quad x\in\Om^c,
\enn
and set $M=20$. The exact and numerical solutions on $\G$ are plotted  in Figure \ref{Fig3} and \ref{Fig4} when $N=64$ and $\omega=1$. We observe that the numerical solutions are in a perfect agreement with the exact ones from the qualitative point of view. In Table \ref{Table1} and \ref{Table2}, we present the numerical errors
\ben
\|\textcolor{rot}{{\bb u}-{\bb u}_h}\|_{(H^{0}(\G))^2}
\enn
with respect to $N$ as $\omega= 3$ and $5$, which confirm the optimal order of accuracy.

\begin{figure}[htbp]
\centering
\begin{tabular}{cc}
\includegraphics[scale=0.35]{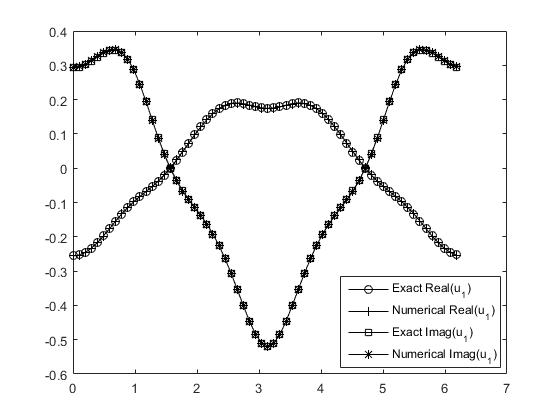} &
\includegraphics[scale=0.35]{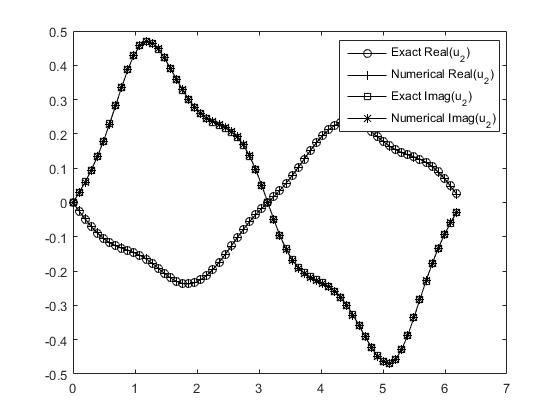}
\end{tabular}
\caption{\textcolor{rot}{The real and imaginary parts of the exact and numerical solutions when $\Om$ is rounded-triangle-shaped for Example 2.}}
\label{Fig3}
\end{figure}

\begin{figure}[htbp]
\centering
\begin{tabular}{cc}
\includegraphics[scale=0.35]{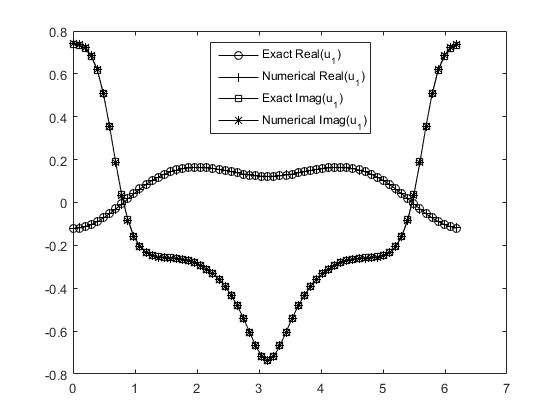} &
\includegraphics[scale=0.35]{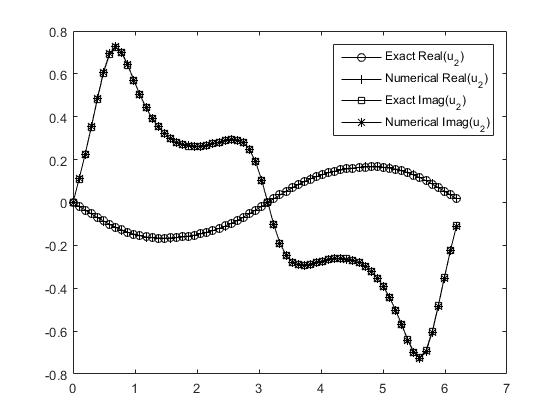}
\end{tabular}
\caption{\textcolor{rot}{The real and imaginary parts of the exact and numerical solutions when $\Om$ is kite-shaped for Example 2.}}
\label{Fig4}
\end{figure}

\begin{table}[htbp]
\caption{Numerical errors vs the number of total elements $N$ when $\Om$ is rounded-triangle-shaped for Example 2.}
\centering
\begin{tabular}{ccccccc}
\hline
$N$ & $\omega$ & $\|\textcolor{rot}{{\bb u}-{\bb u}_h}\|_{(H^{0}(\G))^2}$ & Order & $\omega$ & $\|\textcolor{rot}{{\bb u}-{\bb u}_h}\|_{(H^{0}(\G))^2}$ & Order \\
\hline
 256  &   & 2.05E-3 & --     &   & 6.64E-3 & --    \\
 512  & 3 & 9.58E-4 & 1.10   & 5 & 3.38E-3 & 0.97  \\
1024  &   & 4.84E-4 & 0.99   &   & 1.73E-3 & 0.97  \\
\hline
\end{tabular}
\label{Table1}
\end{table}

\begin{table}[htbp]
\caption{Numerical errors vs the number of total elements $N$ when $\Om$ is kite-shaped for Example 2.}
\centering
\begin{tabular}{ccccccc}
\hline
$N$ & $\omega$ & $\|\textcolor{rot}{{\bb u}-{\bb u}_h}\|_{(H^{0}(\G))^2}$ & Order & $\omega$ & $\|\textcolor{rot}{{\bb u}-{\bb u}_h}\|_{(H^{0}(\G))^2}$ & Order \\
\hline
 256  &   & 1.92E-3 & --     &   & 3.41E-3 & --   \\
 512  & 3 & 8.80E-4 & 1.13   & 5 & 1.58E-3 & 1.11  \\
1024  &   & 4.32E-4 & 1.03   &   & 7.89E-4 & 1.00  \\
\hline
\end{tabular}
\label{Table2}
\end{table}

\begin{figure}[htbp]
\centering
\begin{tabular}{c}
\includegraphics[scale=0.4]{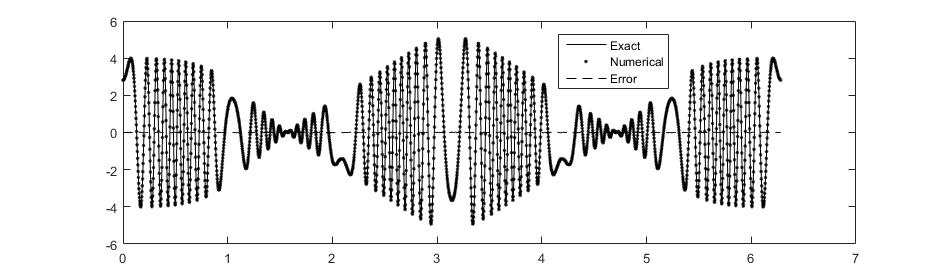} \\
(a) $\mbox{Re}\,u_1$ \\
\includegraphics[scale=0.4]{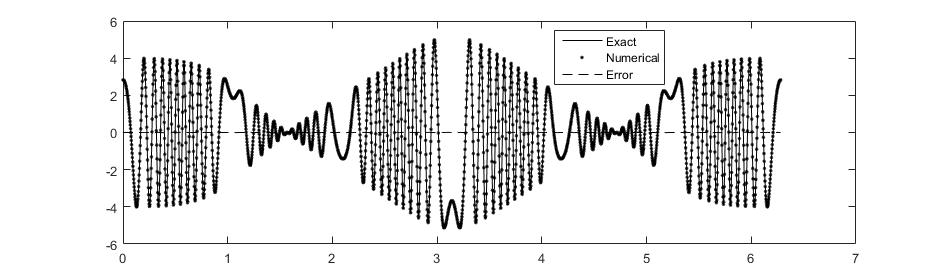} \\
(b) $\mbox{Re}\,u_2$
\end{tabular}
\caption{\textcolor{rot}{The real parts of the exact solution ${\bb u}$, numerical solution ${\bb u}_h$  for Example 3 with $\omega=40\pi$.}}
\label{Fig5}
\end{figure}

\begin{figure}[htbp]
\centering
\begin{tabular}{c}
\includegraphics[scale=0.4]{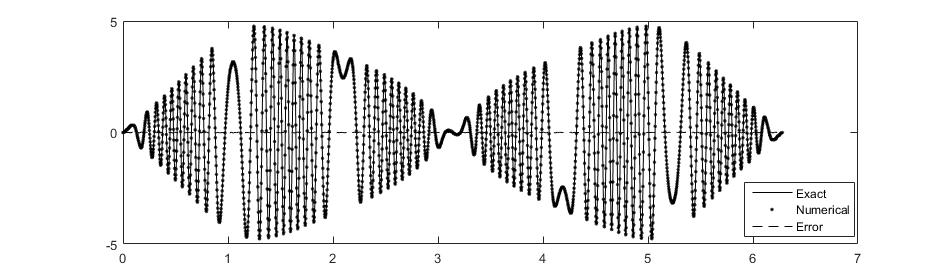} \\
(a) $\mbox{Im}\,u_1$ \\
\includegraphics[scale=0.4]{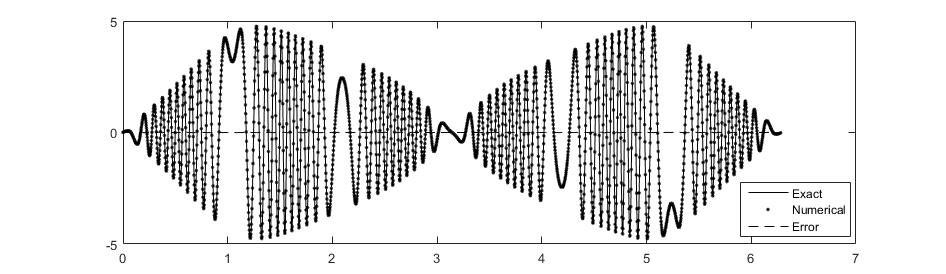} \\
(b) $\mbox{Im}\,u_2$
\end{tabular}
\caption{\textcolor{rot}{The imaginary parts of the exact solution ${\bb u}$, numerical solution ${\bb u}_h$  for Example 3 with $\omega=40\pi$.}}
\label{Fig6}
\end{figure}

{\bf Example 3} ({\it accuracy of BEM with high frequency}). In this example, we consider the rounded-triangle-shaped obstacle and compute the solutions with high frequency. Here we choose $M=20$ and consider $\omega=10\pi$, $20\pi$ and $40\pi$  which means that the corresponding shear wavelengths are $\lambda_s=2\pi/k_s=0.2$, $0.1$ and $0.05$, respectively. The number of nodes $N$ is chosen such that
\ben
N>16\pi R_\G/\lambda_s.
\enn
The exact and numerical solutions as $\omega=40\pi$ are shown in Figure \ref{Fig5} and \ref{Fig6}, from which one can observe that the numerical solutions are also in a perfect agreement with the exact ones from the qualitative point of view. The numerical errors measured in $L^2$-norm and $L^\infty$-norm for all three different frequencies are listed in Table \ref{Table3}.

\begin{table}[htbp]
\caption{Numerical errors for Example 3.}
\centering
\begin{tabular}{cccc}
\hline
$\omega$ & $N$ & $\|\textcolor{rot}{{\bb u}-{\bb u}_h}\|_{(H^{0}(\G))^2}$ & $\|\textcolor{rot}{{\bb u}-{\bb u}_h}\|_{(L^\infty(\G))^2}$ \\
\hline
 10$\pi$ & 630   & 3.76E-2 & 1.70E-2   \\
 20$\pi$ & 1260  & 4.69E-2 & 2.64E-2   \\
 40$\pi$ & 2520  & 6.24E-2 & 3.52E-2   \\
\hline
\end{tabular}
\label{Table3}
\end{table}

{\bf Example 4} ({\it scattering by obstacle with complex geometry}). We consider the scattering of an incident compressional plane wave
\ben
\textcolor{rot}{\bb u}^i=d\,e^{ik_px\cdot d},\quad d=(1,0)^\top.
\enn
The obstacle $\Om$ is star-shaped with the boundary $\G$ characterized by
\ben
x(t)=(1+0.3\cos5t)(\cos t,\sin t),\quad t\in[0,2\pi).
\enn
We compute the total field $\textcolor{rot}{{\bb u}^t={\bb u}+{\bb u}^i}=(u_1^t,u_2^t)^\top$ in $\Om^c$ using (\ref{DirectBRF}) with $N=1024$, and present the numerical results in Figure \ref{Fig7}. Our numerical results show that the multiple scattering effects incurred by the concave portion of the obstacle are accurately captured.

\begin{figure}[htbp]
\centering
\begin{tabular}{cc}
\includegraphics[scale=0.3]{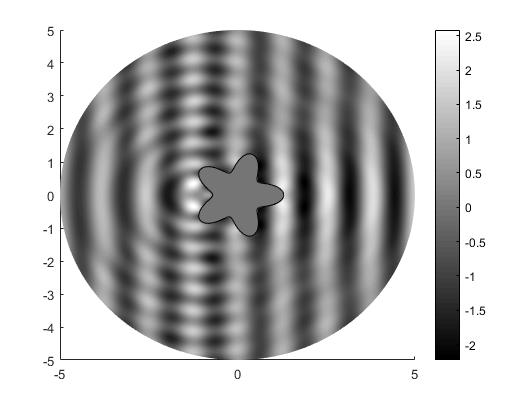} &
\includegraphics[scale=0.3]{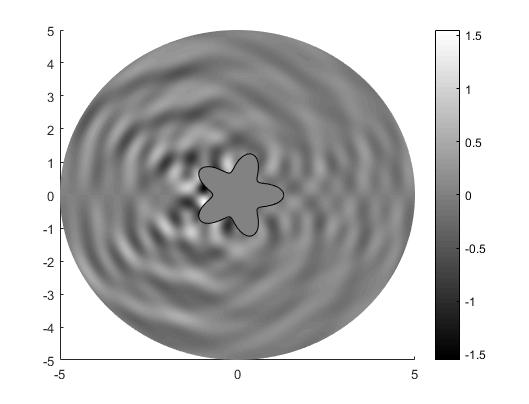} \\
(a) $\mbox{Re}\,u_1^t$ & (b) $\mbox{Re}\,u_2^t$ \\
\includegraphics[scale=0.3]{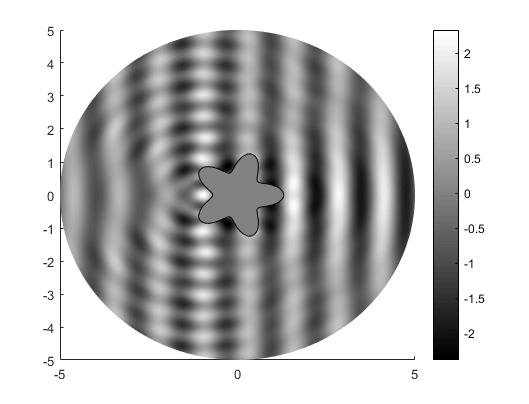} &
\includegraphics[scale=0.3]{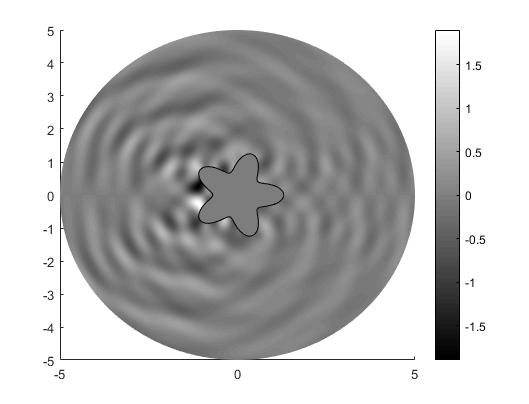} \\
(c) $\mbox{Im}\,u_1^t$ & (d) $\mbox{Im}\,u_2^t$
\end{tabular}
\caption{The real and imaginary parts of the numerical total field when $\Om$ is star-shaped.}
\label{Fig7}
\end{figure}

{\bf Example 5} ({\it scattering by multi-scale obstacles}). \textcolor{rot}{In this example}, we consider the scattering of an incident point source located at the origin by a coupling of an extended kite-shaped obstacle and a relatively small ellipse-shaped obstacle (see Figure 1(d)). The incident wave is selected as
\ben
\textcolor{rot}{\bb u}^{i}=-\nabla\varphi,\quad \varphi(x)=H_0^{(1)}(k_p|x|),\quad x\ne 0.
\enn
We make a comparison between the scattering phenomenon by multi-scale obstacles and that by only extended obstacle. We choose $N=1024$, and present the numerical solutions in Figure \ref{Fig8} and \ref{Fig9}, respectively. It can be seen that, even though the small obstacle has a relatively small effect on the scattered field, this influence has been  indeed captured by our methods, and can be observed from  the mid-left part of $\mbox{Re}\,u_1$ in Figure \ref{Fig8} and \ref{Fig9} for instance.

\begin{figure}[htbp]
\centering
\begin{tabular}{cc}
\includegraphics[scale=0.32]{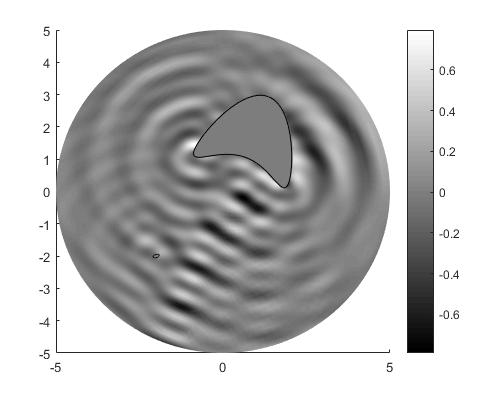} &
\includegraphics[scale=0.32]{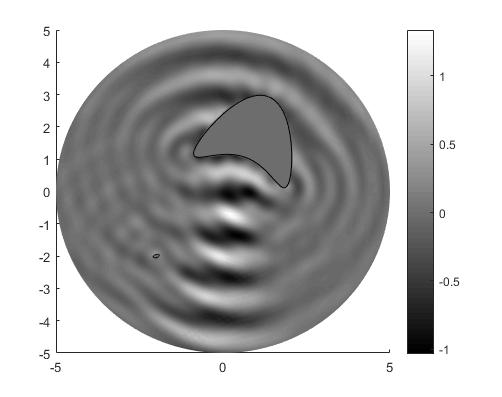} \\
(a) $\mbox{Re}\,u_1$ & (b) $\mbox{Re}\,u_2$ \\
\includegraphics[scale=0.32]{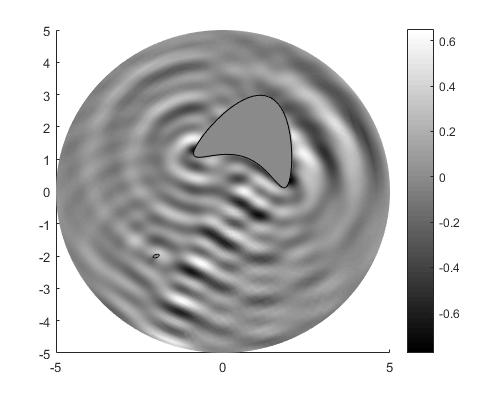} &
\includegraphics[scale=0.32]{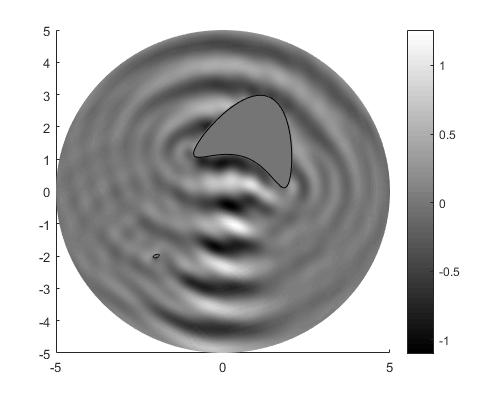} \\
(c) $\mbox{Im}\,u_1$ & (d) $\mbox{Im}\,u_2$
\end{tabular}
\caption{The real and imaginary parts of the numerical scattered field when $\Om$ is mixed-shaped.}
\label{Fig8}
\end{figure}

\begin{figure}[htbp]
\centering
\begin{tabular}{cc}
\includegraphics[scale=0.32]{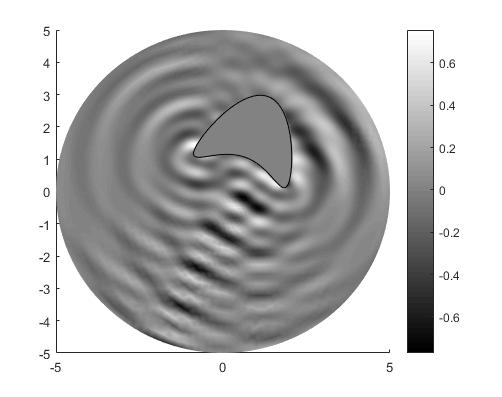} &
\includegraphics[scale=0.32]{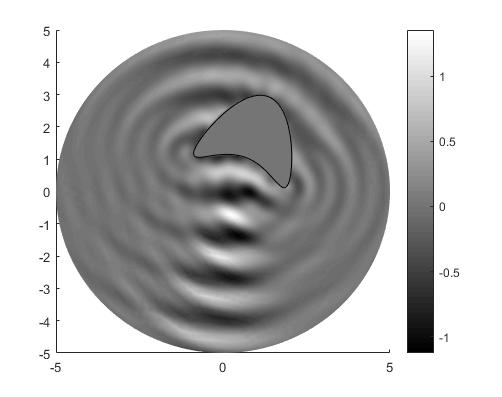} \\
(a) $\mbox{Re}\,u_1$ & (b) $\mbox{Re}\,u_2$ \\
\includegraphics[scale=0.32]{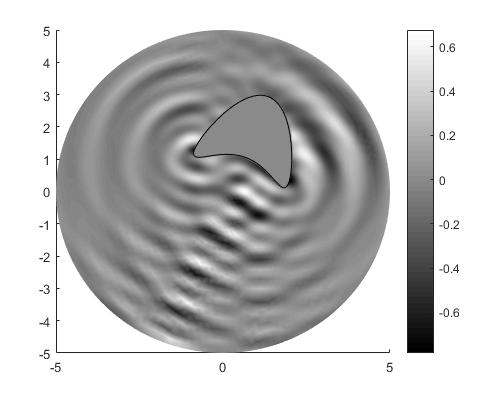} &
\includegraphics[scale=0.32]{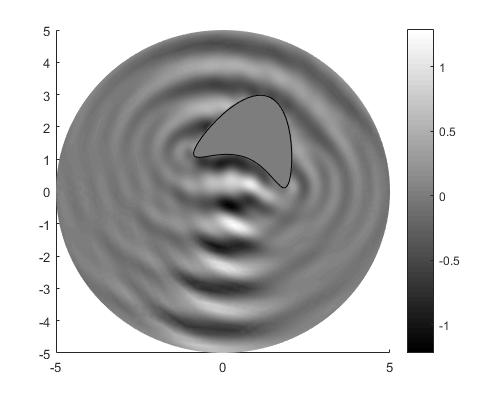} \\
(c) $\mbox{Im}\,u_1$ & (d) $\mbox{Im}\,u_2$
\end{tabular}
\caption{The real and imaginary parts of the numerical scattered field when $\Om$ is kite-shaped.}
\label{Fig9}
\end{figure}

\textcolor{rot}{{\bf Example 6} ({\it scattering by obstacle with non-smooth boundary}). This example is designed to verify the ability of Galerkin BEM to handle rough surface. Assume that the obstacle $\Omega$ is right-angled-triangle-shaped, see Figure \ref{Fig1}(e). We set the exact solution to be the same as in Example 2 and choose $M=20$. The exact and numerical solutions on $\G$ are plotted in Figure \ref{Fig10} when $N=64$ and $\omega=1$ and in Figure \ref{Fig11} when $N=640$ and $\omega=10\pi$. We also observe that the numerical solutions are in a perfect agreement with the exact ones from the qualitative point of view. In Table \ref{Table4}, we present the numerical errors $\|{\bb u}-{\bb u}_h\|_{(H^{0}(\G))^2}$ with respect to $N$ for low frequencies, which give lower order of accuracy compared with the scattering with smooth boundary.}

\begin{figure}[htbp]
\centering
\begin{tabular}{cc}
\includegraphics[scale=0.35]{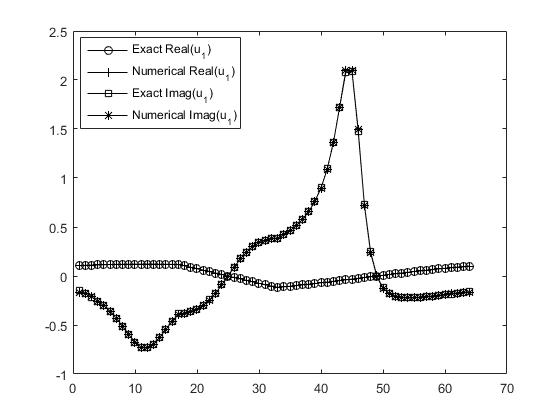} &
\includegraphics[scale=0.35]{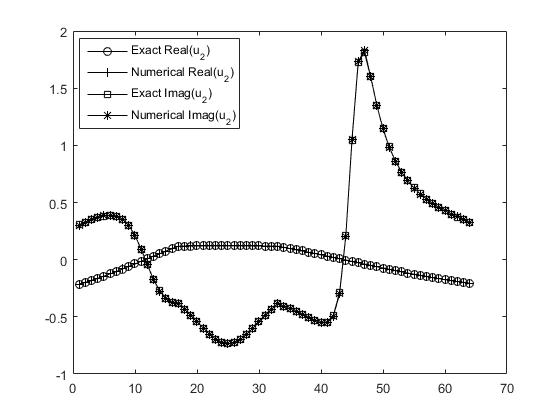}
\end{tabular}
\caption{\textcolor{rot}{The real and imaginary parts of the exact and numerical solutions when $\Om$ is right-angled-triangle-shaped for Example 6 with $N=64$ and $\omega=1$.}}
\label{Fig10}
\end{figure}

\begin{figure}[htbp]
\centering
\begin{tabular}{cc}
\includegraphics[scale=0.35]{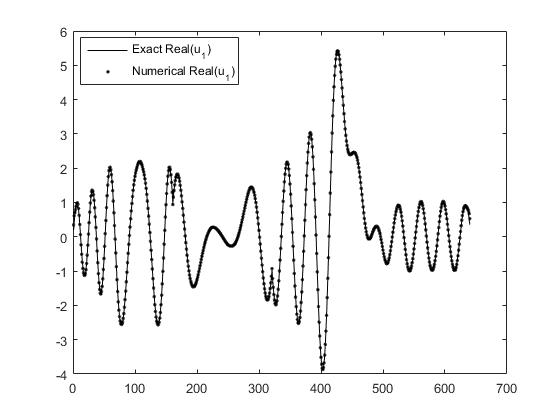} &
\includegraphics[scale=0.35]{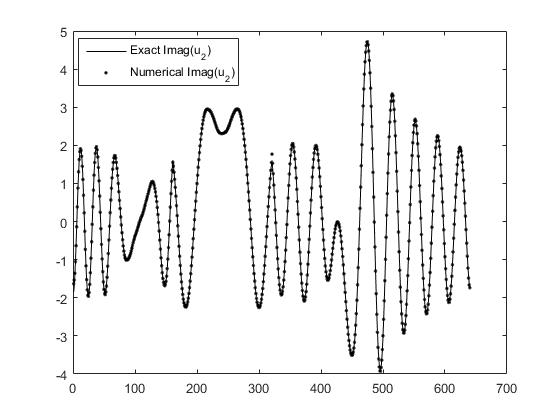}
\end{tabular}
\caption{\textcolor{rot}{The real and imaginary parts of the exact and numerical solutions when $\Om$ is right-angled-triangle-shaped for Example 6 with $N=640$ and $\omega=10\pi$.}}
\label{Fig11}
\end{figure}

\begin{table}[htbp]
\textcolor{rot}{\caption{Numerical errors vs the number of total elements $N$ when $\Om$ is right-angled-triangle-shaped for Example 6.}
\begin{center}
\begin{tabular}{ccccccc}
\hline
$N$ & $\omega$ & $\|\textcolor{rot}{{\bb u}-{\bb u}_h}\|_{(H^{0}(\G))^2}$ & Order & $\omega$ & $\|\textcolor{rot}{{\bb u}-{\bb u}_h}\|_{(H^{0}(\G))^2}$ & Order \\
\hline
 128  &   & 1.66E-2 & --     &   & 3.23E-2 & --    \\
 256  & 3 & 9.24E-3 & 0.85   & 5 & 1.88E-2 & 0.78  \\
 512  &   & 5.39E-3 & 0.78   &   & 1.10E-2 & 0.77  \\
\hline
\end{tabular}
\end{center}}
\label{Table4}
\end{table}

\textcolor{rot}{{\bf Example 7} ({\it scattering with high contrast parameters}). In this last example, we consider the exterior elastic scattering problem with high contrast parameters. Assume that the obstacle $\Omega$ is kite-shaped and the exact solution is set to be the same as in Example 2. We always choose $N=64$ and $\omega=3$. Firstly, we consider the case of small shear modulus, i.e., small $\mu$. We plot the exact and numerical solutions on $\G$ in Figure \ref{Fig12} choosing $\mu=0.1$. Secondly, we consider the incompressible limit case, i.e., $\lambda\rightarrow\infty$. We choose $\lambda=100$ and present the numerical solutions in Figure \ref{Fig13}. The numerical results shows that our methods can also handle the cases of high contrast parameters.}

\begin{figure}[htbp]
\centering
\begin{tabular}{cc}
\includegraphics[scale=0.35]{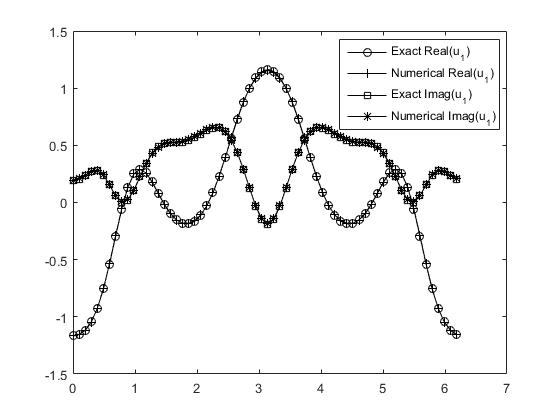} &
\includegraphics[scale=0.35]{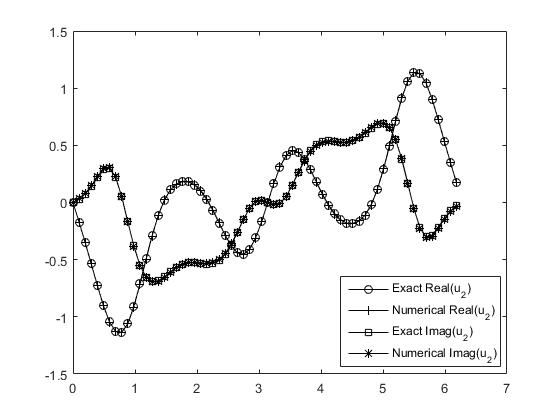}
\end{tabular}
\caption{\textcolor{rot}{The real and imaginary parts of the exact and numerical solutions when $\Om$ is kite-shaped for Example 7 with $\mu=0.1$.}}
\label{Fig12}
\end{figure}

\begin{figure}[htbp]
\centering
\begin{tabular}{cc}
\includegraphics[scale=0.35]{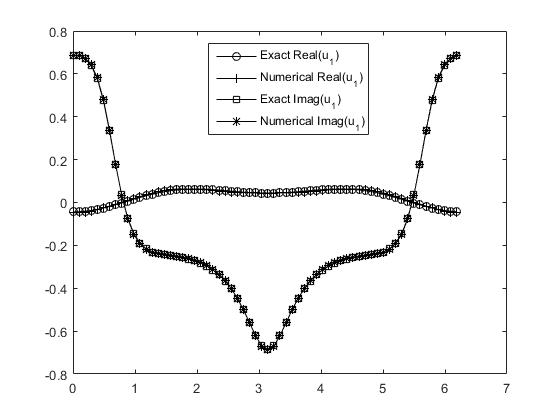} &
\includegraphics[scale=0.35]{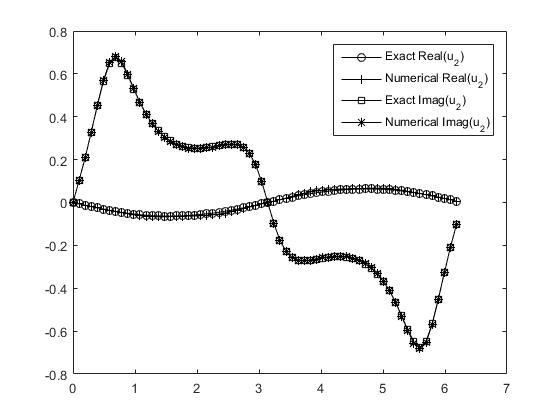}
\end{tabular}
\caption{\textcolor{rot}{The real and imaginary parts of the exact and numerical solutions when $\Om$ is kite-shaped for Example 6 with $\lambda=100$.}}
\label{Fig13}
\end{figure}

\section{Conclusion}
\label{sec:7}
In this work, we developed Galerkin boundary element methods to solve the two dimensional elastic wave scattering problem.  In particular, a novel computational approach is proposed for the evaluation of weakly-singular, singular, and hypersingular boundary integral operators corresponding to time-harmonic Navier equations. Several numerical experiments have been presented to demonstrate efficiency and accuracy of the proposed numerical formulation and methods. Due to the fact that the matrix ${\bf A}_h$ in (\ref{linearsys}) is usually complex and indefinite,  we plan to develop computational techniques  associated with high frequency, fast algorithms, and valid pre-conditioner for the linear system.

\appendix
\section{Computational formulations}

We present the computational formulations of the entries (\ref{entry1})-(\ref{entry3}). Following the notation in Section 4.3, define
\ben
K_0(\G^1,\G^2,k_1,k_2,i,j)=\int_{\G^1}\int_{\G^2} (\textcolor{rot}{{\bb T}_y{\bb E}}(x,y))^\top\varphi_j(y)\varphi_i(x)\,ds_yds_x,
\enn
and
\ben
K'_0(\G^1,\G^2,k_1,i,j)=\int_{\G^1}\int_{\G^2} \textcolor{rot}{{\bb T}_x{\bb E}}(x,y)\varphi_i(x)\,ds_yds_x,
\enn
which lead to
\ben
{\bf K}_{h}(i,j) &=& K_0(\G_{i-1},\G_{j-1},1,1,i,j)+K_0(\G_i,\G_{j-1},-1,1,i,j)\\
&+& K_0(\G_{i-1},\G_j,1,-1,i,j)+K_0(\G_i,\G_j,-1,-1,i,j),\\
{\bf K}'_{h}(i,j) &=& K'_0(\G_{i-1},\G_j,1,-1,i,j)+K'_0(\G_i,\G_j,-1,-1,i,j).
\enn
From the regularized formulation (\ref{Ks}) and (\ref{Ksp}), we know that when $\G^1\ne\G^2$ the Gauss quadrature rule can be used naturally for computing
\be
\label{KSij1}
&\quad& K_0(\G^1,\G^2,k_1,k_2,i,j)\nonumber\\
&=& \frac{ik_s|\G^1||\G^2|}{64} \int_{-1}^1\int_{-1}^1 H_1^{(1)}(k_s|{\bf r}|)(1+k_2\xi_2)(1+k_1\xi_1)\,\frac{{\bf r}^\top{\bf n}_{\G^2}}{|{\bf r}|}\textcolor{rot}{\bb I}\,d\xi_2d\xi_1\nonumber\\
&+& \frac{ik_2|\G^1|}{32} \int_{-1}^1\int_{-1}^1 H_0^{(1)}(k_s|{\bf r}|) (1+k_1\xi_1){\bf N} \,d\xi_2d\xi_1 \nonumber\\
&-& \frac{i|\G^1||\G^2|}{64} \int_{-1}^1\int_{-1}^1  \frac{\left[k_sH_1^{(1)}(k_s|{\bf r}|)-k_pH_1^{(1)}(k_p|{\bf r}|)\right]} {|{\bf r}|} (1+k_2\xi_2)(1+k_1\xi_1){\bf r}\,{\bf n}_{\G^2}^\top\,d\xi_2d\xi_1\nonumber\\
&-& \frac{i\mu k_2|\G^1|}{16\rho\omega^2} \int_{-1}^1\int_{-1}^1  \frac{\left[k_sH_1^{(1)}(k_s|{\bf r}|)-k_pH_1^{(1)}(k_p|{\bf r}|)\right]} {|{\bf r}|}(1+k_1\xi_1){\bf N}\,d\xi_2d\xi_1 \nonumber\\
&+& \frac{i\mu k_2|\G^1|}{16\rho\omega^2} \int_{-1}^1\int_{-1}^1  \frac{\left[k_s^2H_2^{(1)}(k_s|{\bf r}|)-k_p^2H_2^{(1)}(k_p|{\bf r}|)\right]} {|{\bf r}|^2}(1+k_1\xi_1){\bf r}\,{\bf r}^\top{\bf N}\,d\xi_2d\xi_1,
\en
and
\be
\label{KSPij1}
&\quad& K'_0(\G^1,\G^2,k_1,i,j)\nonumber\\
&=& -\frac{ik_s|\G^1||\G^2|}{32} \int_{-1}^1\int_{-1}^1 H_1^{(1)}(k_s|{\bf r}|) (1+k_1\xi_1)\frac{{\bf r}^\top{\bf n}_{\G^1}}{|{\bf r}|}\textcolor{rot}{\bb I}\,d\xi_2d\xi_1\nonumber\\
&-& \frac{ik_1|\G^2|}{16} \int_{-1}^1\int_{-1}^1 H_0^{(1)}(k_s|{\bf r}|) {\bf N}\,d\xi_2d\xi_1  \nonumber\\
&+& \frac{i|\G^1||\G^2|}{32} \int_{-1}^1\int_{-1}^1  \frac{\left[k_sH_1^{(1)}(k_s|{\bf r}|)-k_pH_1^{(1)}(k_p|{\bf r}|)\right]} {|{\bf r}|} (1+k_1\xi_1){\bf n}_{\G^1}{\bf r}^\top\,d\xi_2d\xi_1\nonumber\\
&+& \frac{i\mu k_1|\G^2|}{8\rho\omega^2} \int_{-1}^1\int_{-1}^1  \frac{\left[k_sH_1^{(1)}(k_s|{\bf r}|)-k_pH_1^{(1)}(k_p|{\bf r}|)\right]} {|{\bf r}|}{\bf N}\,d\xi_2d\xi_1 \nonumber\\
&-& \frac{i\mu k_1|\G^2|}{8\rho\omega^2} \int_{-1}^1\int_{-1}^1  \frac{\left[k_s^2H_2^{(1)}(k_s|{\bf r}|)-k_p^2H_2^{(1)}(k_p|{\bf r}|)\right]} {|{\bf r}|^2} {\bf N} {\bf r}\,{\bf r}^\top\,d\xi_2d\xi_1.
\en
Further, we have for $\G^1=\G^2$ that
\be
\label{KSij2}
&\quad& K_0(\G^1,\G^2,k_1,k_2,i,j)\nonumber\\
&=& \sum_{m=0}^\infty \frac{ik_s^{2m}|\G^1|^{2m+1}k_2} {2^{2m+5}} \left[\left(C_m^5+C_m^6\ln\frac{k_s|\G^1|}{4}\right)I_m^3 +C_m^6I_m^4\right] {\bf N} \nonumber\\
&-& \sum_{m=0}^\infty \frac{i|\G^1|^{2m+3}(k_1-k_2)} {2^{2m+7}} \left[\left(C_m^1+C_m^2\ln\frac{|\G^1|}{2}\right)I_m^1 +C_m^2I_m^2\right] {\bf t}_{\G^1}{\bf n}_{\G^1}^\top \nonumber\\
&-& \sum_{m=0}^\infty \frac{i\mu|\G^1|^{2m+1}k_2} {2^{2m+4}\rho\omega^2} \left[\left(C_m^1+C_m^2\ln\frac{|\G^1|}{2}\right)I_m^3 +C_m^2I_m^4\right] {\bf N} \nonumber\\
&+& \sum_{m=0}^\infty \frac{i\mu|\G^1|^{2m+3}k_2} {2^{2m+6}\rho\omega^2} \left[\left(C_m^3+C_m^4\ln\frac{|\G^1|}{2}\right)I_{m+1}^3 +C_m^4I_{m+1}^4\right] {\bf t}_{\G_1}\,{\bf t}_{\G_1}^\top{\bf N} \nonumber\\
&+& \frac{\mu(k_s^2-k_p^2)k_2}{4\pi\rho\omega^2}{\bf t}_{\G_1}\,{\bf t}_{\G_1}^\top{\bf N},
\en
and
\be
\label{KSPij2}
&\quad& K'_0(\G^1,\G^2,k_1,i,j)\nonumber\\
&=& -\sum_{m=0}^\infty \frac{ik_s^{2m}|\G^1|^{2m+1}k_1} {2^{2m+4}} \left[\left(C_m^5+C_m^6\ln\frac{k_s|\G^1|}{4}\right)I_m^3 +C_m^6I_m^4\right] {\bf N} \nonumber\\
&+& \sum_{m=0}^\infty \frac{i|\G^1|^{2m+3}k_1} {2^{2m+6}} \left[\left(C_m^1+C_m^2\ln\frac{|\G^1|}{2}\right)I_m^1 +C_m^2I_m^2\right] {\bf n}_{\G^1}{\bf t}_{\G^1}^\top \nonumber\\
&+& \sum_{m=0}^\infty \frac{i\mu|\G^1|^{2m+1}k_1} {2^{2m+3}\rho\omega^2} \left[\left(C_m^1+C_m^2\ln\frac{|\G^1|}{2}\right)I_m^3 +C_m^2I_m^4\right] {\bf N} \nonumber\\
&-& \sum_{m=0}^\infty \frac{i\mu|\G^1|^{2m+3}k_1} {2^{2m+5}\rho\omega^2} \left[\left(C_m^3+C_m^4\ln\frac{|\G^1|}{2}\right)I_{m+1}^3 +C_m^4I_{m+1}^4\right] {\bf N}{\bf t}_{\G_1}\,{\bf t}_{\G_1}^\top \nonumber\\
&-& \frac{\mu|\G^1|(k_s^2-k_p^2)k_1}{2\pi\rho\omega^2}{\bf N}{\bf t}_{\G_1}\,{\bf t}_{\G_1}^\top.
\en
Now we consider the matrix ${\bf W}_{h}$. It follows that
\ben
{\bf W}_{h}(i,j) &=& \int_{\tilde{\G}} (W\varphi_j)\varphi_i\,ds\\
&=& \textcolor{rot}{-}\int_{\G_{i-1}\cup\G_i}\int_{\G_{j-1}\cup\G_j} \textcolor{rot}{\bb T}_x (\textcolor{rot}{\bb T}_y\textcolor{rot}{\bb E}(x,y))^\top\varphi_j(y)\varphi_i(x)\,ds_yds_x.
\enn
Then it suffices to evaluate the integral
\ben
W_0(\G^1,\G^2,k_1,k_2,i,j)=\textcolor{rot}{-}\int_{\G^1}\int_{\G^2} \textcolor{rot}{\bb T}_x (\textcolor{rot}{\bb T}_y\textcolor{rot}{\bb E}(x,y))^\top\varphi_j(y)\varphi_i(x)\,ds_yds_x,
\enn
which further implies that
\ben
{\bf W}_{h}(i,j) &=& W_0(\G_{i-1},\G_{j-1},1,1,i,j)+W_0(\G_i,\G_{j-1},-1,1,i,j)\\
&+& W_0(\G_{i-1},\G_j,1,-1,i,j)+W_0(\G_i,\G_j,-1,-1,i,j)
\enn
We know from (\ref{Ws}) that
\be
\label{WSij1}
&\quad& W_0(\G^1,\G^2,k_1,k_2,i,j) \nonumber\\
&=& \mu k_s^2\int_{\G^1}\int_{\G^2} \gamma _{k_s}(x,y)\varphi_j(y)\varphi_i(x)\left({\bf n}_{\G^1}{\bf n}_{\G^2}^\top-{\bf n}_{\G^1}^\top{\bf n}_{\G^2}\textcolor{rot}{\bb I}+{\bf N}{\bf n}_{\G^1}^\top{\bf t}_{\G^2}\right)\,ds_yds_x\nonumber\\
&-& \mu k_s^2\int_{\G^1}\int_{\G^2} \gamma _{k_p}(x,y)\varphi_j(y)\varphi_i(x){\bf n}_{\G^1}{\bf n}_{\G^2}^\top\,ds_yds_x\nonumber\\
&+& 4\mu^2\int_{\G^1}\int_{\G^2}\textcolor{rot}{\bb E}(x,y)\frac{\varphi_j(y)}{ds_y} \frac{\varphi_i(x)}{ds_x}\,ds_yds_x\nonumber\\
&-& \frac{4\mu^2}{\lambda+2\mu}\int_{\G^1}\int_{\G^2} \gamma_{k_p}(x,y)\frac{\varphi_j(y)}{ds_y} \frac{\varphi_i(x)}{ds_x}\textcolor{rot}{\bb I}\,ds_yds_x\nonumber\\
&+& 2\mu\int_{\G^1}\int_{\G^2} {\bf n}_{\G^1} \nabla _{x}^\top R(x,y)\frac{\varphi_j(y)}{ds_y}\varphi_i(x){\bf N}\,ds_yds_x\nonumber\\
&-& 2\mu\int_{\G^1}\int_{\G^2} {\bf N}\nabla_{y} R(x,y){\bf n}_{\G^2}^\top\varphi_j(y)\frac{\varphi_i(x)}{ds_x}\,ds_yds_x.
\en
Then
\be
\label{WSij2}
&\quad& W_0(\G^1,\G^2,k_1,k_2,i,j) \nonumber\\
&=& \frac{i\mu k_s^2|\G^1||\G^2|}{64}\int_{-1}^1\int_{-1}^1 H_0^{(1)}(k_s|{\bf r}|)(1+k_2\xi_2)(1+k_1\xi_1)\,d\xi_2d\xi_1\nonumber\\
&\times& \left({\bf n}_{\G^1}{\bf n}_{\G^2}^\top+\begin{bmatrix}
-{\bf n}_{\G^1}^\top{\bf n}_{\G^2}&-{\bf n}_{\G^1}^\top{\bf t}_{\G^2}\\
{\bf n}_{\G^1}^\top{\bf t}_{\G^2}&-{\bf n}_{\G^1}^\top{\bf n}_{\G^2}
\end{bmatrix}\right)\nonumber\\
&-& \frac{i\mu k_s^2|\G^1||\G^2|}{64}\int_{-1}^1\int_{-1}^1 H_0^{(1)}(k_p|{\bf r}|)(1+k_2\xi_2)(1+k_1\xi_1)\,{\bf n}_{\G^1}{\bf n}_{\G^2}^\top\,d\xi_2d\xi_1\nonumber\\
&+& \frac{i\mu k_1k_2}{4}\int_{-1}^1\int_{-1}^1 H_0^{(1)}(k_s|{\bf r}|) \textcolor{rot}{\bb I}\,d\xi_2d\xi_1\nonumber\\
&-& \frac{i\mu^2k_1k_2}{4(\lambda+2\mu)}\int_{-1}^1\int_{-1}^1 H_0^{(1)}(k_p|{\bf r}|)\textcolor{rot}{\bb I}\,d\xi_2d\xi_1\nonumber\\
&-& \frac{i\mu^2k_1k_2}{4\rho\omega^2}\int_{-1}^1\int_{-1}^1 \frac{\left[k_sH_1^{(1)}(k_s|{\bf r}|)-k_pH_1^{(1)}(k_p|{\bf r}|)\right]} {|{\bf r}|}\textcolor{rot}{\bb I}\,d\xi_2d\xi_1\nonumber\\
&-& \frac{i\mu k_2|\G^1|}{16}\int_{-1}^1\int_{-1}^1  \frac{\left[k_sH_1^{(1)}(k_s|{\bf r}|)-k_pH_1^{(1)}(k_p|{\bf r}|)\right]} {|{\bf r}|}
(1+k_1\xi_1){\bf n}_{\G^1}\,{\bf r}^\top{\bf N}\,d\xi_2d\xi_1\nonumber\\
&-& \frac{i\mu k_1|\G^2|}{16}\int_{-1}^1\int_{-1}^1  \frac{\left[k_sH_1^{(1)}(k_s|{\bf r}|)-k_pH_1^{(1)}(k_p|{\bf r}|)\right]} {|{\bf r}|}
(1+k_2\xi_2){\bf N}{\bf r}\,{\bf n}_{\G^2}^\top\,d\xi_2d\xi_1\nonumber\\
&+& \frac{i\mu^2k_1k_2}{4\rho\omega^2}\int_{-1}^1\int_{-1}^1 \frac{\left[k_s^2H_2^{(1)}(k_s|{\bf r}|)-k_p^2H_2^{(1)}(k_p|{\bf r}|)\right]} {|{\bf r}|^2} {\bf r}\,{\bf r}^\top\,d\xi_2d\xi_1
\en
If $\G^1=\G^2$, by Lemma \ref{series} we have
\be
\label{WSij3}
&\quad& W_0(\G^1,\G^2,k_1,k_2,i,j) \nonumber\\
&=& \sum_{m=0}^\infty \begin{pmatrix}
\frac{i\mu k_s^{2m+2}|\G^1|^{2m+2}}{2^{2m+6}}\left({\bf n}_{\G^1}{\bf n}_{\G^1}^\top-\textcolor{rot}{\bb I}\right)\times\\
\left[\left(C_m^5+C_m^6\ln\frac{k_s|\G^1|}{4}\right)(I_m^3+k_1k_2I_m^5) +C_m^6(I_m^4+k_1k_2I_m^6)\right]
\end{pmatrix}\nonumber\\
&-& \sum_{m=0}^\infty \begin{pmatrix}
\frac{i\mu k_s^2k_p^{2m}|\G^1|^{2m+2}}{2^{2m+6}} {\bf n}_{\G^1}{\bf n}_{\G^1}^\top\times\\
\left[\left(C_m^5+C_m^6\ln\frac{k_p|\G^1|}{4}\right)(I_m^3+k_1k_2I_m^5) +C_m^6(I_m^4+k_1k_2I_m^6)\right]
\end{pmatrix}\nonumber\\
&+& \sum_{m=0}^\infty \frac{i\mu k_s^{2m}|\G^1|^{2m}k_1k_2}{2^{2m+2}} \left[\left(C_m^5+C_m^6\ln\frac{k_s|\G^1|}{4}\right)I_m^3 +C_m^6I_m^4\right]\textcolor{rot}{\bb I} \nonumber\\
&-& \sum_{m=0}^\infty \frac{i\mu^2 k_p^{2m}|\G^1|^{2m}k_1k_2}{2^{2m+2}(\lambda+2\mu)} \left[\left(C_m^5+C_m^6\ln\frac{k_p|\G^1|}{4}\right)I_m^3 +C_m^6I_m^4\right]\textcolor{rot}{\bb I} \nonumber\\
&-& \sum_{m=0}^\infty \frac{i\mu^2 |\G^1|^{2m}k_1k_2}{2^{2m+2}\rho\omega^2} \left[\left(C_m^1+C_m^2\ln\frac{|\G^1|}{2}\right)I_m^3 +C_m^2I_m^4\right] \textcolor{rot}{\bb I}\nonumber\\
&-& \sum_{m=0}^\infty \frac{i\mu |\G^1|^{2m+2}k_1k_2}{2^{2m+5}} \left[\left(C_m^1+C_m^2\ln\frac{|\G^1|}{2}\right)I_m^1 +C_m^2I_m^2\right] {\bf n}_{\G^1} {\bf t}_{\G_1}^\top{\bf N}\nonumber\\
&+& \sum_{m=0}^\infty \frac{i\mu |\G^1|^{2m+2}k_1k_2}{2^{2m+5}} \left[\left(C_m^1+C_m^2\ln\frac{|\G^1|}{2}\right)I_m^1 +C_m^2I_m^2\right]
{\bf N}{\bf t}_{\G_1}\,{\bf n}_{\G^1}^\top\nonumber\\
&+& \sum_{m=0}^\infty \frac{i\mu^2 |\G^1|^{2m+2}k_1k_2}{2^{2m+4}\rho\omega^2} \left[\left(C_m^3+C_m^4\ln\frac{|\G^1|}{2}\right)I_{m+1}^3 +C_m^4I_{m+1}^4\right]{\bf t}_{\G_1}\,{\bf t}_{\G_1}^\top \nonumber\\
&+& \frac{\mu^2(k_s^2-k_p^2)k_1k_2}{\pi\rho\omega^2}{\bf t}_{\G_1}\,{\bf t}_{\G_1}^\top.
\en
Finally, the matrixes ${\bf I}_{1h}$ and ${\bf I}_{2h}$ can be evaluated as follows:
\be
\label{I1ij}
{\bf I}_{1h}(i,j) &=& \int_{\tilde{\G}} \varphi_j(x)\varphi_i(x)\,ds_x \textcolor{rot}{\bb I}\nonumber\\
&=& \left[\frac{|\G_{i-1}|}{6}\delta_{i-1,j} +\left(\frac{|\G_{i-1}|}{3} +\frac{|\G_{i}|}{3}\right)\delta_{i,j} +\frac{|\G_{i}|}{6}\delta_{i+1,j}\right] \textcolor{rot}{\bb I},
\en
and
\be
\label{I2ij}
{\bf I}_{2h}(i,j) &=& \int_{\tilde{\G}} \psi_j(x)\varphi_i(x)\,ds_x \textcolor{rot}{\bb I}\nonumber\\
&=& \left[\frac{|\G_{j}|}{2}\delta_{i,j} +\frac{|\G_{j}|}{2}\delta_{i,j+1}\right] \textcolor{rot}{\bb I}.
\en

\section*{Acknowledgments}
\textcolor{rot}{The work of G. Bao is supported in part by a NSFC Innovative Group Fun (No.11621101), an Integrated Project of the Major Research Plan of NSFC (No. 91630309), and an NSFC A3 Project (No. 11421110002). The work of L. Xu is partially supported by a Key Project of the Major Research Plan of NSFC (No. 91630205), and a NSFC Grant (No. 11371385).} The work of T. Yin is partially supported by the NSFC Grant (No. 11371385). \textcolor{rot}{We would also like to thank anonymous referees for their
comments.}

\end{document}